\documentclass[sn-mathphys-num]{sn-jnl}

\usepackage{graphicx} 
\usepackage{amsfonts}
\usepackage{amsmath,booktabs,threeparttable,boxedminipage}
\usepackage{amsthm}%
\usepackage{mathrsfs}%
\usepackage[title]{appendix}%
\usepackage{xcolor}%
\usepackage{textcomp}%
\usepackage{manyfoot}%
\usepackage{algorithm}%
\usepackage{algorithmicx}%
\usepackage{algpseudocode}%
\usepackage{listings}%
\usepackage{qbordermatrix}
\usepackage[noabbrev,nameinlink]{cleveref}
\usepackage{float}
\usepackage{lineno}
\usepackage[caption=false]{subfig} 

\crefname{subsection}{subsection}{subsections}  
\crefformat{equation}{\textcolor{blue}{#2\textcolor{blue}{(#1)}#3}}

\algrenewcommand\algorithmicrequire{\textbf{Input:}}
\algrenewcommand\algorithmicensure{\textbf{Output:}}

\newcommand\argmin{\mathop{\mathrm{argmin}}}
\newcommand\calR{\mathcal{R}}
\newcommand\calN{\mathcal{N}}

\newcommand\calX{\mathcal{X}}

\newcommand\calA{\mathcal{A}}

\newcommand\calP{\mathcal{P}}
\newcommand\calS{\mathcal{S}}

\newcommand\rI{\mathrm{I}}
\newcommand\sT{\mathsf{T}}

\numberwithin{equation}{section}
\numberwithin{figure}{section}
\numberwithin{table}{section}
\newtheorem{theorem}{Theorem}[section]
\newtheorem{lemma}{Lemma}[section]
\newtheorem{proposition}{Proposition}[section]

 \newtheorem{corollary}{Corollary}[section]

\begin{document}

\title[LSE]{Krylov iterative methods for linear least squares problems with linear equality constraints}


\author*[1]{\fnm{Haibo} \sur{Li}}\email{haibo.li@unimelb.edu.au}



\affil*[1]{\orgdiv{School of Mathematics and Statistics}, \orgname{The University of Melbourne}, \orgaddress{\street{Parkville}, \city{Melbourne}, \postcode{3010}, \state{VIC}, \country{Australia}}}




\abstract{
	We consider the linear least squares problem with linear equality constraints (LSE problem) formulated as $\min_{x\in\mathbb{R}^{n}}\|Ax-b\|_2 \ \mathrm{s.t.} \ Cx = d$. Although there are some classical methods available to solve this problem, most of them rely on matrix factorizations or require the null space of $C$, which limits their applicability to large-scale problems. To address this challenge, we present a novel analysis of the LSE problem from the perspective of operator-type least squares (LS) problems, where the linear operators are induced by $\{A,C\}$. We show that the solution of the LSE problem can be decomposed into two components, each corresponding to the solution of an operator-form LS problem. Building on this decomposed-form solution, we propose two Krylov subspace based iterative methods to approximate each component, thereby providing an approximate solution of the LSE problem. Several numerical examples are constructed to test the proposed iterative algorithm for solving the LSE problems, which demonstrate the effectiveness of the algorithms.
}



\maketitle

\section{Introduction}\label{sec1}
The linear least squares problem with equality constraints (LSE problem) arises frequently in various fields such as data fitting, signal processing, control systems and optimization \cite{damm2013linear,fear2004visualizing,zhu2007recursive,pisinger2007bivariate}. These problems involve minimizing a least squares objective function while ensuring that a set of linear equality constraints is satisfied. The general formulation of the LSE problem is formulated as:
\begin{equation}\label{LSE0}
	\min_{x\in\mathbb{R}^{n}}\|Ax-b\|_2 \ \ \mathrm{s.t.} \ \ 
	Cx = d,
   \end{equation}
where $A\in\mathbb{R}^{m\times n}$, and $C\in\mathbb{R}^{p\times n}$. It restricts the solution space to the set of solutions that satisfy both the least squares objective and the linear equality constraints, which is often used in cases where certain relationships between the variables are known a priori and must be preserved. The LSE problem \cref{LSE0} has a solution if and only if $Cx=d$ is consistent, and it has a unique solution if and only if $(A^{\sT},C^{\sT})^{\sT}$ has full column rank. There is a large amount of work on the analysis of the LSE problem; see e.g. \cite{wei1992algebraic,gulliksson1995backward,ding1998new,gulliksson2000perturbation,gulliksson2002perturbation}.

Despite their wide applicability, solving large-scale LSE problems efficiently remains a significant computational challenge. Classical solution approaches typically reduce the constrained LSE problem to an equivalent unconstrained problem by eliminating the constraints. The key strategy of these methods are the constraint substitution technique, which eliminates the constraints by reducing the dimension of the problem. 
The first one is usually called the null space method \cite{hanson1969extensions,lawson1995solving,schittkowski1978factorization,stoer1971numerical}. This method involves finding a null space basis for the matrix $C$ using a rank-revealing QR factorization \cite{chan1987rank,chan1992some,gu1996efficient}. The constraints are then incorporated into the LS problem by substituting this basis into the system, leading to a reduced, unconstrained problem of lower dimension. This approach provides numerical stability and is widely used in many practical settings. The second one is usually called the direct elimination method \cite{lawson1995solving}. In this method, a substitution is made directly by expressing certain solution components (those affected by the constraints) in terms of others. This can be accomplished using a pivoted LU factorization or a rank-revealing QR factorization of $C$ \cite{Golub2013}. The direct elimination method exhibits good numerical stability and efficiency, particularly when implemented with appropriate matrix factorizations.

In addition to constraint substitution methods, there are some other methods that transform the constrained LS problem to an unconstrained optimization problem. The method based on the Lagrange multiplier formulation \cite{golub1965numerical,heath1982some,Golub2013} is often useful. This approach introduces auxiliary variables (Lagrange multipliers) to incorporate the constraints into the optimization process, which constructs an augmented system by combining the linear constraints and the LS problem, and both can be solved simultaneously. This method provides a powerful and general way to enforce equality constraints during the optimization. Techniques like weighting and updating procedures can also be used to enforce constraints progressively, ensuring that the solution satisfies the constraints a posteriori \cite{van1985method,barlow1988direct,gulliksson1994iterative,stewart1997weighting}.

All the above methods, when implemented correctly, can provide a solution with satisfied accuracy. However, in many practical scenarios, the problem size can be very large. In such cases, matrix factorization-based methods become impractical due to their cubic scaling computational complexity. This highlights the need to develop new iterative methods for solving the LSE problem that do not rely on matrix factorizations. The Krylov subspace method is well-known for its effectiveness in solving linear systems, including linear equations and LS problems, where only matrix-vector multiplications are required during the iteration process \cite{liesen2013krylov, Golub2013}. However, up  to now, there is a lack of Krylov iterative methods specifically for the LSE problem, possibly due to an incomplete understanding of its properties. Establishing connections between the LSE and LS problems could be valuable, as it would aid in the development of efficient Krylov iterative methods for solving the LSE problem.

In this paper, we present a novel analysis of the LSE problem from the perspective of operator-type LS problems. Building on this framework, we propose two Krylov subspace based iterative methods for solving LSE problems. To this end, we construct two linear operators using the matrices $\{A,C\}$ and formulate two LS problems associated with these operators. Using these formulations, we investigate the structure of the solutions to the LSE problem and show that its minimum 2-norm solution can be decomposed into two components, each corresponding to the solution of one of the operator-based LS problems. Building on this connection, we derive two types of decomposed-form solution for the LSE problem. To approximate the solution, it is sufficient to solve the associated operator-form LS problems using the Golub-Kahan bidiagonalization process \cite{Golub1965, caruso2019convergence, li2024new, li2024characterizing}. This approach leads to Krylov subspace-based iterative procedures. Consequently, we develop two Krylov iterative methods for the LSE problem, each corresponding to solving one of the decomposed-form solutions. The proposed algorithms do not rely on any matrix factorizations. Instead, they follow an inner-outer iteration structure, where, at each outer iteration, an inner subproblem is approximately solved. We also propose a procedure for constructing LSE problems for testing purposes and present several numerical examples to illustrate the effectiveness of the proposed algorithms.
 
The paper is organized as follows. In \Cref{sec2}, we review three commonly used methods for the LSE problem. In \Cref{sec3}, we analyze the LSE problem from the perspective of operator-type LS problems and derive two types of decomposed-form solution. In \Cref{sec4} we proposed two Krylov subspace based iterative algorithms for approximating the decomposed-form solution. Numerical experiments are presented in \Cref{sec5}, and concluding remarks follow in \Cref{sec6}.

Throughout the paper, we denote by $\calN(\cdot)$ and $\calR(\cdot)$ the null space and range space of a matrix or linear operator, respectively, denote by $\rI$ and $\mathbf{0}$ the identity matrix and zero matrix/vector with orders clear from the context, and denote by $\mathrm{span}\{\cdot\}$ the subspace spanned by a group of vectors or columns of a matrix. We use $\calP_{\calS}$ to denote the orthogonal operator onto a closed subspace $\calS$.

\section{LSE problem and its computation}\label{sec2}
We review three classical methods for the LSE problem: the null space approach, the method of direct elimination, and the augmented system approach. 

The null space method was developed and discussed by a number of authors in the 1970s. The basic idea is that any vector $x\in\mathbb{R}^{n}$ satisfying the linear constraint $Cx=d$ can be written as $x=x_0+Zy$, where $x_0$ is a particular solution of $Cx=d$, and the columns of $Z\in\mathbb{R}^{n\times t}$ form a basis for $\calN(C)$. Let the QR factorization of $C$ be  
\begin{equation}
	CP = Q \bordermatrix*[()]{%
	R & \mathbf{0} & p \cr
	p & n-p  \cr
},
\end{equation}
where $P \in \mathbb{R}^{n \times n}$ is a permutation matrix representing the pivoting, $R$ is an upper triangular matrix and $Q$ is an orthogonal matrix. Now we can get a solution of $Cx=d$:
\begin{equation}
	x_0 = P \bordermatrix*[()]{%
	R^{-1}Q^{\sT}d  &  p \cr
	\mathbf{0}  &  \ \ n-p  \cr
	 &
	}.
\end{equation}
Now the LSE problem \cref{LSE0} becomes 
\begin{equation}\label{eq:transformed}
    \min_{y\in\mathbb{R}^{t}} \left\| AZy - (b -A x_0) \right\|.
\end{equation}
By solving the above standard LS problem to get the solution $y^{\dag}$, we get a solution $x^{\dag}=x_0+Zy^{\dag}$ to the LSE problem.

In the null space method, the matrix $Q$ should be stored explicitly or implicitly (by using e.g, Householder transformations), leading to a relatively high memory demands and implied operation counts. The more challenging point is that the matrix $Z$ is usually dense, which makes it inefficient to solve the LS problem \cref{eq:tranformed mx}. Also,  In recent years, there are some works about constructing a sparse null space matrix $Z$, where the QR factorization of $C$ with a threshold pivoting is used; see \cite{scott2022null,scott2022solving}. 

The second method is the direct elimination, which involves expressing the dependence of the selected $p$ components of the vector $x$ on the remaining $n-p$ components, and this relationship is then substituted into the LS problem in \cref{LSE0}. Suppose $P\in\mathbb{R}^{n\times n}$ is a permutation matrix such that $CP=\begin{pmatrix}
	C_1 \ C_2
\end{pmatrix}$ with $C_1\in\mathbb{R}^{p\times p}$ be a nonsingular matrix. Let 
\begin{equation}
	AP = \bordermatrix*[()]{%
	A_1 & A_2 & m \cr
	p & m-p  \cr
}, \quad \quad
x = Py = \bordermatrix*[()]{%
y_1  &  p \cr
y_2  &  \ \ n-p  \cr
&
}.
\end{equation}
Now we have the substitution $y_1=C_{1}^{-1}(d-C_2y_2)$. Combining this expression with the LS problem in \cref{LSE0}, we have the transformed LS problem
\begin{equation}\label{eq:ls2}
    \min_{y_2\in\mathbb{R}^{n-p}} \left\| \widetilde{A} y_2 -(b - A_1C_1^{-1}d)\right\|_2,
\end{equation}
where 
\begin{equation}\label{eq:tranformed mx}
	\widetilde{A} = A_2 - A_1C_1^{-1}C_2 \in \mathbb{R}^{m \times (n-p)}.
\end{equation}

Once we have the solution $y_2$, then we can compute $y_1$ and finally get the solution of \cref{LSE0} with the expression $x=P\begin{pmatrix}
	y_1 \\ y_2
\end{pmatrix}$.
To get $P$ and $C_1$, usually a QR factorization of $C$ with pivoting should be exploited. For sparse matrices $A$ and $C$, some strategies have been proposed to make that the transformed matrix $\widetilde{A}$ has some sparse structure \cite{scott2017solving,scott2022solving}, leading to a sparse LS problem \cref{eq:tranformed mx} that can be computed effectively by an iterative solver.

The third method is the augmented system method, which is based on the method of Lagrange multiplier for constrained optimization problem. Consider the following Lagrangian function for the constrained LS problem \cref{LSE0}:
\begin{equation}
	f(x,\lambda) = \frac{1}{2}\|Ax-b\|_{2}^{2}+\lambda^{\sT}(d-Cx), \quad \lambda\in\mathbb{R}^{p}.
\end{equation}
Finding the zero root of $\nabla_{x}f(x,\lambda)$ leads to 
\begin{equation*}
	A^{\sT}Ax-A^{\sT}b-C^{\sT}\lambda = \mathbf{0}.
\end{equation*}
By letting $r=b-Ax$ and using $Cx=d$, we have the following symmetric indefinite linear system:
\begin{equation}\label{LSE_aug}
	\begin{pmatrix}
		\mathbf{0} & A^{\sT} & B^{\sT} \\
		A & \rI & \mathbf{0} \\
		B & \mathbf{0} & \mathbf{0}
	\end{pmatrix}
	\begin{pmatrix}
		x \\ r \\ \lambda
	\end{pmatrix}=
	\begin{pmatrix}
		\mathbf{0} \\ b \\ d
	\end{pmatrix} .
\end{equation}
If $A$ and $C$ are sparse and have full rank, then \cref{LSE_aug} is a $(m+n+p)\times(m+n+p)$ sparse nonsingular linear system. Based on the above framework, there are several variants of practical algorithms. We do not discuss them in more details, but refer the readers to \cite{bjorck1984general,zhdanov2012method,zhdanov2015solving,scott2022computational,scott2022solving}.

\section{Decomposed-form solution of the LSE problem}\label{sec3}
In this section, we investigate the structure of the solutions of \cref{LSE0} and derive two decomposed-form expressions of the minimum 2-norm solution of \cref{LSE0}. We consider a more general case, which is formulated as
\begin{equation}\label{LSE2}
 \min_{x\in\calS}\|Ax-b\|_2, \quad
 \mathcal{S} = \{x\in\mathbb{R}^{n}: \|Cx-d\|_{2}=\mathrm{min} \} .
\end{equation}
In this paper, we also call \cref{LSE2} the LSE problem. Note that if $Cx=d$ is a consistent linear system, then \cref{LSE2} is equivalent to \cref{LSE0}. In the rest part of the paper, we focus on the analysis and computation of \cref{LSE2}.

The following theorem about the generalized linear least squares (GLS) problem will be used in the subsequent analysis. We refer to \cite{elden1982weighted,li2024new} for more details. 
\medskip
\begin{theorem}\label{thm:GLS_solv}
	For any $K\in\mathbb{R}^{m\times n}$ and $L\in\mathbb{R}^{p\times n}$, consider the GLS problem
	\begin{equation}\label{GLS}
		  \min_{x\in\mathbb{R}^{n}}\|Lx\|_{2} \ \ \ \mathrm{s.t.} \ \ \
		  \|Kx-g\|_{2}=\mathrm{min} .
	\end{equation}
	The following properties hold:
	\begin{enumerate}
		\item[(1)] a vector $x\in\mathbb{R}^{n}$ is a solution of \cref{GLS} if and only if
		\begin{equation}\label{GLS_criterion}
		  \begin{cases}
			  K^{\sT}(Kx-b)=\mathbf{0} , \\
			  x^{\sT}Mz = 0 , \ \ \ \forall \ z \in \mathcal{N}(K) ,
		  \end{cases}
		\end{equation}
		where $M=K^{\sT}K+L^{\sT}L$;
		\item[(2)] there exist a unique solution in $\calR(M)$, which is the minimum 2-norm solution of \cref{GLS}, given by $x=K_{L}^{\dag}g$, where $K_{L}^{\dag}:=(\rI-(L\calP_{\calN(K)})^{\dag}L)K^{\dag}$ is the weighted pseudoinverse of $K$;
		\item[(3)] define the linear operator 
		\begin{equation}
			T: \calX:=(\calR(M), \langle\cdot,\cdot\rangle_{M}) \rightarrow (\mathbb{R}^{m},\langle\cdot,\cdot\rangle_{2}), \ \ \ 
		v \mapsto Kv ,
		\end{equation}
		where $v$ and $Kv$ are column vectors under the canonical bases of $\mathbb{R}^{n}$ and $\mathbb{R}^{m}$. Then the minimum $\|\cdot\|_{\calX}$-norm solution of the least squares problem
		\begin{equation}\label{operator_ls}
			\min_{v\in\calX} \|Tv - b\|_{2}
		\end{equation} is the minimum 2-norm solution of \cref{GLS}.
	\end{enumerate}
\end{theorem}

The following result characterizes the structure of the solutions of \cref{LSE2}.
\medskip
\begin{theorem}\label{LSE_criterion}
	Let $G=A^{\sT}A+C^{\sT}C$. The minimum 2-norm solution of \cref{LSE2} is
	\begin{equation}\label{decomp1}
		x^{\dag} = C_{A}^{\dag}d + (\calP_{\calR(G)}-C_{A}^{\dag}C)A^{\dag}b ,
	\end{equation}
	and the set of all the solutions is $x^{\dag}+\calN(G)$.
\end{theorem}
\begin{proof}
	First note that $\calN(G)=\calN(A)\cap\calN(C)$. Thus, if $x$ is a solution of \cref{LSE2}, then $\calP_{\cal{R}(G)}x$ is also a solution. Conversely, if $x\in\calR(G)$ is a solution, then $x+z$ is also a solution for any $z\in\calN(G)$. Since $\calN(G)\perp\calR(G)$, the minimum 2-norm solution of \cref{LSE2} must in $\calR(G)$. 
	
	Notice that
	\begin{equation*}
		\|Ax-b\|_{2}^{2}=\|Ax-\calP_{\calR(A)}b\|_{2}^{2}+\|\calP_{\calR(A)^{\perp}}b\|_{2}^{2}
	\end{equation*}
	and the second term is independent of $x$. Therefore, we can rewrite \cref{LSE2} as 
	\begin{equation*}
		\min_{x\in\mathbb{R}^{n}}\|A(x-A^{\dag}b)\|_2, \ \ \ \mathrm{s.t.} \ \ \
 		\|Cx-d\|_{2}=\mathrm{min}.
	\end{equation*}
	Using the transformation $\tilde{x}=x-A^{\dag}b$ and noticing that $Cx-d=C(\tilde{x}+A^{\dag}b)-d=C\tilde{x}-(d-CA^{\dag}b)$, the above problem is equivalent to
	\begin{equation}\label{GLS2}
		\min_{\tilde{x}\in\mathbb{R}^{n}}\|A\tilde{x}\|_2, \ \ \ \mathrm{s.t.} \ \ \
 		\|C\tilde{x}-(d-CA^{\dag}b)\|_{2}=\mathrm{min}.
	\end{equation}
	By \Cref{thm:GLS_solv}, the general solution of this problem is 
	\begin{equation*}
		\tilde{x}=C_{A}^{\dag}(d-CA^{\dag}b) + z, \ \ \ z\in\calN(G) .
	\end{equation*}
	Therefore, the general solution of \cref{LSE2} is 
	\begin{equation}
		x = A^{\dag}b + C_{A}^{\dag}(d-CA^{\dag}b) + z
		  = C_{A}^{\dag}d + (I_{n}-C_{A}^{\dag}C)A^{\dag}b + z, \ \ \ z\in\calN(G) .
	\end{equation}
	Note from \Cref{thm:GLS_solv} that $\calR(C_{A}^{\dag})\subseteq\calR(G)$, which indicates that the projection of the above solution onto $\calR(G)$ is $x^{\dag}$. Thus, $x^{\dag}$ is a solution of \cref{LSE2} in $\calR(G)$. 
	
	It only remains to show that there exists a unique solution of \cref{LSE2} in $\calR(G)$. To see it, notice from the above transformation that $x\in\calR(G)$ is a solution of \cref{LSE2} if and only if $x-\calP_{\calR(G)}A^{\dag}b\in\calR(G)$ is a solution of \cref{GLS2}. By \Cref{thm:GLS_solv}, \cref{GLS2} has a unique solution in $\calR(G)$, this implies that \cref{LSE2} has a unique in $\calR(G)$.
\end{proof}

Write
\begin{equation}
	x_{1}^{\dag} = C_{A}^{\dag}d, \quad x_{2}^{\dag} =  (\calP_{\calR(G)}-C_{A}^{\dag}C)A^{\dag}b.
\end{equation}
The minimum 2-norm solution of \cref{LSE2} has the decomposed-form: $x^{\dag}=x_{1}^{\dag}+x_{2}^{\dag}$. Note that $x_{1}^{\dag}$ is the minimum 2-norm solution of the GLS problem
\begin{equation}\label{GLS_CA}
	\min_{x\in\mathbb{R}^{n}}\|Ax\|_{2} \ \ \ \mathrm{s.t.} \ \ \
		  \|Cx-d\|_{2}=\mathrm{min} .
\end{equation}
By \Cref{thm:GLS_solv}, $x_{1}^{\dag}$ is also the solution of the operator-form LS problem \cref{operator_ls}, where $K=C$ and $L=A$. We can use the iterative method proposed in \cite{li2024new} to approximate $x_{1}^{\dag}$. Although the expression of $x_{2}^{\dag}$ looks relatively complicated, the following result shows that it is the minimum 2-norm solution of an LS problem with a null space constraint.
\medskip
\begin{theorem}\label{thm:x2}
	Let $x_{2}^{\dag}=(\calP_{\calR(G)}-C_{A}^{\dag}C)A^{\dag}b$, then $x_{2}^{\dag}$ is the minimum 2-norm solution of
	\begin{equation}\label{NS_LS}
		\min_{x\in\calN(C)}\|Ax-b\|_{2} .
	\end{equation}
\end{theorem}

The following lemma is needed for the proof.
\medskip
\begin{lemma}\label{lem:NSLS}
	A vector $x\in\calN(C)$ is the minimum 2-norm solution of \cref{NS_LS} if and only if
	\begin{equation*}\label{NSLS_criterion}
		\begin{cases}
			\calP_{\calN(C)}(A^{\sT}(Ax-b))=\mathbf{0} , \\
			x \perp \calN(A)\cap\calN(C) .
		\end{cases}
	  \end{equation*}
\end{lemma}
\begin{proof}
	Define the linear operator
	\begin{equation}\label{op_A}
		\calA: (\calN(C), \langle \cdot, \cdot \rangle_{2}) \rightarrow (\mathbb{R}^{m}, \langle \cdot, \cdot \rangle_{2}),  \ \ \ 
		v \mapsto Av,
	\end{equation}
	where $v$ and $Av$ are column vectors under the canonical bases of $\mathbb{R}^{n}$ and $\mathbb{R}^{m}$. Notice that $\calX:=(\calN(C), \langle \cdot, \cdot \rangle_{2})$ is a finite dimensional Hilbert space. Therefore, there exist a unique minimum $\calX$-norm solution of $\min_{v\in\calX}\|\calA v-b\|_{2}$, which is the minimum 2-norm solution of \cref{NS_LS}, and $x\in\calN(C)$ is the minimum $\calX$-norm solution if and only if
	\begin{equation*}
		\calA^{*}(\calA x-b)=\mathbf{0}, \ \ \ 
		x \perp_{\calX} \calN(\calA)
	\end{equation*}
	where the orthogonal relation $\perp_{\calX}$ in $\calX$ is the 2-orthogonal relation in $\calN(C)$, and the linear operator $\calA^{*}: (\mathbb{R}^{m}, \langle \cdot, \cdot \rangle_{2}) \rightarrow (\calN(C), \langle \cdot, \cdot \rangle_{2})$ is the adjoint of $\calA$ defined by the relation $\langle \calA v, u \rangle_{2}=\langle v, \calA^{*}u \rangle_{2}$ for any $v\in\calN(C)$ and $u\in\mathbb{R}^{m}$. It is easy to verify that $\calA^{*}v=\calP_{\calN(C)}A^{\sT}v$ under the canonical bases. Thus, $\calA^{*}(\calA x-b)=\mathbf{0}$ is equivalent to $\calP_{\calN(C)}(A^{\sT}(Ax-b))=\mathbf{0}$. Since $\calN(\calA)=\{x\in\calN(C): Ax=0\}=\calN(A)\cap\calN(C)$, it follows that $x \perp_{\calX} \calN(\calA)$ is equivalent to $x \perp \calN(A)\cap\calN(C)$.
\end{proof} 

Now we can prove \Cref{lem:NSLS}.
\begin{proof}[Proof of \Cref{lem:NSLS}]
	The proof contains three steps. 

	Step 1: prove $x_{2}^{\dag}\in\calN(C)$. Using \cite[Theorem 3.7]{li2024new}, we have the relation $CC_{A}^{\dag}C=C$. It follows that 
	\begin{align*}
		Cx_{2}^{\dag} 
		= (C\calP_{\calR(G)} - CC_{A}^{\dag}C)A^{\dag}b 
		= C(\rI-\calP_{\calR(G)})A^{\dag}b
		= C\calP_{\calN(G)}A^{\dag}b = \mathbf{0},
	\end{align*}
	where we have used $\calN(G)\subseteq \calN(C)$.

	Step 2: prove $\calP_{\calN(C)}(A^{\sT}(Ax_{2}^{\dag}-b))=\mathbf{0}$. First we have
	\begin{align*}
		A^{\sT}(Ax_{2}^{\dag}-b)
		&= A^{\sT} [A\calP_{\calR(G)}A^{\dag}b-b-AC_{A}^{\dag}CA^{\dag}b] \\
		&= A^{\sT}(AA^{\dag}-\rI)b - A^{\sT}AC_{A}^{\dag}CA^{\dag}b \\
		&= - A^{\sT}AC_{A}^{\dag}CA^{\dag}b,
	\end{align*}
	where we have used $A\calP_{\calR(G)}x=Ax-A\calP_{\calN(G)}x=Ax$ for any $x\in\mathbb{R}^{n}$, and $\rI-AA^{\dag}=\calP_{\calR(A)^{\perp}}=\calP_{\calN(A^{\sT})}$. Let $w=C_{A}^{\dag}CA^{\dag}b$. By \Cref{LSE_criterion}, $w$ is the minimum 2-norm solution of 
	\begin{equation*}
		\min\|Ax\|_{2} \ \ \ \mathrm{s.t.} \ \ \
		\|Cx-CA^{\dag}b\|_{2}=\min .
	\end{equation*}
	Using \Cref{LSE_criterion} again, it follows that $w^{\sT}Gz=0$ for any $z\in\calN(C)$, which is just
	\begin{equation*}
		w^{\sT}(A^{\sT}A+C^{\sT}C)z = (A^{\sT}Aw)^{\sT}z = 0
	\end{equation*}
	for any $z\in\calN(C)$, which means that $A^{\sT}Aw\perp \calN(C)$. This proves $\calP_{\calN(C)}A^{\sT}Aw=0$, which is the desired result.

	Step 3: prove $x_{2}^{\dag} \perp \calN(A)\cap\calN(C)$. This is obvious by noticing that $x_{2}^{\dag}\in\calR(G)$ and $\calR(G)\perp \calN(A)\cap\calN(C)$. 
\end{proof}

From the above proof, we know that $x_{2}^{\dag}$ is the minimum $\calX$-norm solution of operator-form LS problem $\min_{\calX}\|\calA x-b\|_2$ with $\calA$ defined in \cref{op_A}. Therefore, we have $x_{2}^{\dag}=\calA^{\dag}b =:A_{\calN(C)}^{\dag}b$. Note that $A_{\calN(C)}^{\dag}$ is essentially the matrix form of $\calA^{\dag}$ under the canonical bases of $\mathbb{R}^{n}$ and $\mathbb{R}^{m}$, which depends both on $A$ and $\calN(C)$. 

Based on \Cref{thm:x2}, we will propose an iterative method for the LS problem \cref{NS_LS} to approximate $x_{2}^{\dag}$. Before this, let us investigate several properties of the matrix $A_{\calN(C)}^{\dag}$, which will be used to derive another decomposed-form solution of \cref{LSE2}.
\medskip
\begin{proposition}\label{prop:2pi}
	The following two equalities hold: 
	\begin{equation}
		\begin{cases}
			(\rI-A_{\calN(C)}^{\dag}A)C^{\dag} = C_{A}^{\dag} \\
			(\calP_{\calR(C)}-C_{A}^{\dag}C)A^{\dag} = A_{\calN(C)}^{\dag} .
		\end{cases}
	\end{equation}
\end{proposition}
\begin{proof}
	The second equality is directly derived from \Cref{thm:x2}. Now we prove the first equality.
	By \Cref{thm:GLS_solv}, for any $y\in\mathbb{R}^{n}$, $C_{A}^{\dag}y$ is the 2-minimum solution of 
	\begin{equation*}
		\min\|Ax\|_{2} \ \ \ \mathrm{s.t.} \ \ \ \|Cx-y\|_{2}=\min ,
	\end{equation*}
	which has the same solution as 
	\begin{equation*}
		\min\|Ax\|_{2} \ \ \ \mathrm{s.t.} \ \ \ \|C(x-C^{\dag}y)\|_{2}=\min .
	\end{equation*}
	Let $\bar{x}=x-C^{\dag}y$. The above problem becomes
	\begin{equation*}
		\min\|A\bar{x}+AC^{\dag}y\|_{2}  \ \ \ \mathrm{s.t.} \ \ \ \|C\bar{x}\|_{2}=\min,
	\end{equation*}
	which has the minimum 2-norm solution 
	$\bar{x}^{\dag}=-A_{\calN(C)}^{\dag}AC^{\dag}y$, and a general solution is $\bar{x}=\bar{x}^{\dag}+z$ with $z\in\calN(A)\cap\calN(C)$. Therefore a general solution of the original problem is 
	\[x = C^{\dag}y + \bar{x}^{\dag}y+z .\]
	Note that $\calP_{\calN(C)}C^{\dag}y=(\rI-C^{\dag}C)C^{\dag}y=\mathbf{0}$. Thus, $C^{\dag}y\perp \calN(C)$ and $C^{\dag}y\perp z$. Combining with $\bar{x}_{\dag}\perp z$ we have $C^{\dag}y + \bar{x}^{\dag}\perp z$. Therefore, $C^{\dag}y + \bar{x}^{\dag}y=(\rI-A_{\calN(C)}^{\dag}A)C^{\dag}y$ is the minimum 2-norm solution  of the original problem. Since $y$ is arbitrary, we finally get $(\rI-A_{\calN(C)}^{\dag}A)C^{\dag} = C_{A}^{\dag}$.
\end{proof}

From the above result, we obtain the following decomposed-form solution of \cref{LSE2}.
\medskip
\begin{corollary}\label{cor:lse}
	The minimum 2-norm solution of \cref{LSE2} has the form
	\begin{equation}\label{decomp2}
		x^{\dag} = C^{\dag}d + A_{\calN(C)}^{\dag}(b-AC^{\dag}d)
	\end{equation}
\end{corollary}
\begin{proof}
	Using \Cref{thm:GLS_solv} and \Cref{prop:2pi}, we have
	\begin{align*}
		x^{\dag} &=  C_{A}^{\dag}d + A_{\calN(C)}^{\dag}b
		= (\rI-A_{\calN(C)}^{\dag}A)C^{\dag}d+ A_{\calN(C)}^{\dag}b \\
		&= A_{\calN(C)}^{\dag}(b-AC^{\dag}d)+C^{\dag}d,
	\end{align*}
	which is the desired result.
\end{proof}

By \Cref{LSE_criterion} and \Cref{cor:lse}, we can give two approaches for computing $x^{\dag}$.
\paragraph{The first approach.}
\begin{enumerate}
	\item[(1)] Solve the GLS problem \cref{GLS_CA} to get $x_{1}^{\dag}=C_{A}^{\dag}d$;
	\item[(2)] Solve the LS problem \cref{NS_LS} to get $x_{2}^{\dag}=A_{\calN(C)}^{\dag}b$;
	\item[(3)] Compute $x^{\dag}=x_{1}^{\dag}+x_{2}^{\dag}$.
\end{enumerate}

\paragraph{The second approach.}
\begin{enumerate}
	\item[(1)] Solve the LS problem $\min_{x}\|Cx-d\|_2$ to get the minimum 2-norm solution $\tilde{x}_{1}^{\dag}=C^{\dag}d$;
	\item[(2)] Let $\tilde{b}=b-A\tilde{x}_{1}^{\dag}$. Solve the LS problem $\min_{x\in\calN(C)}\|Ax-\tilde{b}\|_{2}$ to get the minimum 2-norm solution $\tilde{x}_{2}^{\dag}=A_{\calN(C)}^{\dag}\tilde{b}$;
	\item[(3)] Compute $x^{\dag}=\tilde{x}_{1}^{\dag}+\tilde{x}_{2}^{\dag}$.
\end{enumerate}

In the next section, we will propose two Krylov subspace based iterative methods for solving \cref{LSE2}, which correspond to the above two approaches, respectively.

\section{Krylov iterative methods for the LSE problem}\label{sec4}
From the previous section, we find that for solving the LSE problem, we need to compute $C_{A}^{\dag}$ or $A_{\calN(C)}^{\dag}$. We first propose the iterative methods for such computations based on the Krylov subspace, then we give two iterative algorithms for the LSE problem.

\subsection{Iterative method for computing $C_{A}^{\dag}$}

Based on \Cref{thm:GLS_solv}, the author in \cite{li2024new} proposes a Krylov iterative method for approximating $C_{A}^{\dag}d$ for a vector $d\in\mathbb{R}^{p}$. The idea is to apply the Golub-Kahan bidiagonalization (GKB) to solve the operator-form LS problem 
\begin{equation}\label{LS_CA}
	\min_{x\in\calX}\|Tx-d\|_{2},
\end{equation}
where $\calX=(\calR(G),\langle\cdot,\cdot\rangle_{G})$ and $T: \calX\rightarrow (\mathbb{R}^{m},\langle\cdot,\cdot\rangle_{2}), \ x\mapsto Cx$ under the canonical bases. Applying the GKB to $\{T,d\}$ we get the recursive relations 
\begin{equation}\label{GKB_op}
	\begin{cases}
		\beta_1 u_1 = d  \\
		\alpha_{i}v_i = T^{*}u_i -\beta_i v_{i-1}  \\
		\beta_{i+1}u_{i+1} = T v_{i} - \alpha_i u_i,
	\end{cases}
\end{equation}
where $T^{*}: (\mathbb{R}^{m},\langle\cdot,\cdot\rangle_{2}) \rightarrow \calX$ is the adjoint operator of $T$ defined by the relation $\langle Tx, y \rangle_{2}=\langle x, T^{*}y \rangle_{G}$ for any $x\in\calX$ and $y\in\mathbb{R}^{m}$. It has been shown in \cite{li2024new} that the matrix form of $T^{*}$ is $G^{\dag}C$. The positive scalars $\alpha_i$ and $\beta_i$ are computed such that $\|v_i\|_{\calX}=\|u_{i}\|_{2}=1$. Note that $v_{0}:=\mathbf{0}$ for the initial step. 

After $k$ steps, the above GKB process generates two Krylov subspaces and projects the LS problem \cref{LS_CA} onto the Krylov subspaces to get a $k$-dimensional LS problem. The solution of the $k$-dimensional LS problem can be updated step by step from the previous one, which converges to $C_{A}^{\dag}d$ as $k$ increases. This leads to the following \Cref{alg:glsqr} for iteratively approximating $C_{A}^{\dag}d$. Please refer to \cite{li2024new} for more details.

\begin{algorithm}[htb]
	\caption{Generalized LSQR (\textsf{gLSQR}) for computing $C_{A}^{\dag}d$}\label{alg:glsqr}
	\begin{algorithmic}[1]
		\Require $A\in\mathbb{R}^{m\times n}$, $C\in\mathbb{R}^{p\times n}$, $d\in\mathbb{R}^{p}$
		\State Compute $\beta_1=\|d\|_2$, \ $u_1=d/\beta_1$ $\beta_1\tilde{u}_1=b$
		\State Compute $s=G^{\dag}C^{\sT}u_{1}$, \ $\alpha_{1}=(s^{\sT}Gs)^{1/2}$, \ $v_{1}=s/\alpha_1$  \Comment{$G=A^{\sT}A+C^{\sT}C$}
    	\State Set $x_0=\mathbf{0}$, $w_1=v_1$, $\bar{\phi}_{1}=\beta_1$, $\bar{\rho}_1=\alpha_1$
		\For {$i=1,2,\dots$ until convergence,}
		\State $r=Cv_{i}-\alpha_{i}u_{i}$
		\State $\beta_{1+1}=\|r\|_{2}$, \ $u_{i+1}=r/\beta_{i+1}$
		\State $s=G^{\dag}C^{\sT}u_{i+1}-\beta_{i+1}v_{i}$
		\State $\alpha_{i+1}=(s^{\sT}Gs)^{1/2}$, \ $v_{i+1}=s/\alpha_{i+1}$
		\State $\rho_{i}=(\bar{\rho}_{i}^{2}+\beta_{i+1}^{2})^{1/2}$
		\State $c_{i}=\bar{\rho}_{i}/\rho_{i}$
    \State $s_{i}=\beta_{i+1}/\rho_{i}$
		\State$\theta_{i+1}=s_{i}\alpha_{i+1}$ 
    \State $\bar{\rho}_{i+1}=-c_{i}\alpha_{i+1}$
		\State $\phi_{i}=c_{i}\bar{\phi}_{i}$ 
    \State $\bar{\phi}_{i+1}=s_{i}\bar{\phi}_{i}$
		\State $x_{i}=x_{i-1}+(\phi_{i}/\rho_{i})w_{i} $
		\State $w_{i+1}=v_{i+1}-(\theta_{i+1}/\rho_{i})w_{i}$
		\EndFor
	\Ensure Approximation to $C_{A}^{\dag}d$
	\end{algorithmic}
\end{algorithm}

In \Cref{alg:glsqr}, the main computational bottleneck is the need to compute $G^{\dag}(C^{\sT}u_i)$ at each iteration. For large-scale matrices, it is generally impractical to obtain $G^{\dag}$ directly. In this case, using the relation
\begin{equation}\
	G^{\dag}(C^{\sT}u_i) = \argmin_{x\in\mathbb{R}^{n}}\|Gx-C^{\sT}u_i\|_2,
\end{equation}
we can compute $G^{\dag}(C^{\sT}u_i)$ by iteratively solving the above LS problem. Furthermore, by noticing that $G^{\dag}(C^{\sT}u_i)$ is the minimum 2-norm solution of the LS problem
\begin{equation}\label{ls1_gkb}
	\min_{x\in\mathbb{R}^{n}}\left\|\begin{pmatrix}
		C \\ A
	\end{pmatrix}x-\begin{pmatrix}
		u_i \\ \mathbf{0}
	\end{pmatrix}\right\|_{2} ,
\end{equation}
we can use the LSQR algorithm \cite{Paige1982} to approximate $G^{\dag}(C^{\sT}u_i)$ without explicitly forming $G$. If $\begin{pmatrix}
	C \\ A
\end{pmatrix}$ is sparse and its sparse QR factorization is not difficult to compute, then we can compute the solution of \cref{ls1_gkb} directly.

\subsection{Iterative method for computing $A_{\calN(C)}^{\dag}$}
Now we consider how to design a GKB based method to approximate $A_{\calN(C)}b$ for a $b\in\mathbb{R}^{m}$. First, suppose an orthonormal basis of the null space $\calN(C)$ is $\{w_{1},\dots,w_{t}\}$. Let $W_t=(w_{1},\dots,w_{t})\in\mathbb{R}^{n\times t}$. Using \Cref{thm:x2}, if follows that $A_{\calN(C)}^{\dag}b=W_{t}f$, where $f\in\mathbb{R}^{t}$ is the minimum 2-norm solution of the LS problem
\begin{equation}\label{LS_AW}
	\min_{f\in\mathbb{R}^{t}}\|(AW_{t})f-b\|_2
\end{equation}
To solve \cref{LS_AW} iteratively, we apply the GKB to $\{AW_{t}, b\}$, which leads to the following recursive relations:
\begin{equation}\label{GKB_op1}
	\begin{cases}
		\delta_{1}p_1 = b  \\
		\gamma_{i}\tilde{q}_i = (AW_{t})^{\sT}p_i -\delta_i \tilde{q}_{i-1}  \\
		\delta_{i+1}p_{i+1} = (AW_{t})\tilde{q}_{i} - \gamma_i p_i,
	\end{cases}
\end{equation}
where the positive scalars are computed such that $\|p_i\|_{2}=\|q_{i}\|_{2}=1$, and we set $q_{0}:=\mathbf{0}$ for the initial step. 

Using the property of GKB, after $k$ steps, it generates two groups of 2-orthonormal vectors $\{p_{i}\}_{i=1}^{k+1}$ and $\{\tilde{q}_i\}_{i=1}^{k+1}$. Then we can approximate the solution of \cref{LS_AW} in the subspace $\mathrm{span}\{\tilde{q}_i\}_{i=1}^{k}$ as $k$ grows from $1$ to $t$. This approach is equivalent to applying the standard LSQR algorithm to \cref{LS_AW}. Therefore, to get a good approximation to $A_{\calN(C)}^{\dag}b$, we can search a solution of \cref{NS_LS} in the subspace $W_{t}\cdot\mathrm{span}\{\tilde{q}_i\}_{i=1}^{k}=\mathrm{span}\{W_{t}\tilde{q}_i\}_{i=1}^{k}$ at the $k$-th iteration. Let $q_{i}=W_{t}\tilde{q}_i$. Note that $\calP_{\calN(C)}=W_{t}W_{t}^{\sT}$. From the recursions \cref{GKB_op1}, we get 
\begin{equation}\label{GKB_op2}
	\begin{cases}
		\delta_{1}p_1 = b  \\
		\gamma_{i}q_i = \calP_{\calN(C)}A^{\sT}p_i -\delta_i q_{i-1}  \\
		\delta_{i+1}p_{i+1} = Aq_{i} - \gamma_i p_i,
	\end{cases}
\end{equation}
where $\|q_i\|_2=1$. The following result demonstrates that this iterative process is essentially an operator-type GKB.
\medskip
\begin{proposition}\label{prop:GKB-op}
	Let the linear operator defined as \cref{op_A}. Then the iterative process \cref{GKB_op2} is equivalent to the GKB applied to $\{\calA,b\}$.
\end{proposition}
\begin{proof}
	From the proof of \Cref{lem:NSLS} we know that $\calA^{*}v=\calP_{\calN(C)}A^{\sT}v$ for any $v\in\mathbb{R}^{m}$ under the canonical bases. Therefore, the second recursive relation in \cref{GKB_op2} is equivalent to $\gamma_{i}q_i = \calA^{*}p_i -\delta_i q_{i-1}$. Now we can find that \cref{GKB_op2} is just the recursions of the operator-type GKB applied to $\{\calA,b\}$ under the canonical bases.
\end{proof}

\Cref{prop:GKB-op} implies that the outputs of the above iterative process do not depend on the choice of 2-orthonormal basis of $\calN(C)$, i.e. it will generate the same vectors $\{p_i,q_i\}$ and scalars $\{\gamma_i,\delta_i\}$ whenever the 2-orthonormal basis $\{w_{i}\}_{i=1}^{t}$ for $\calN(C)$ is used. 

Now we can give the practical computational approach of this iterative process. Since all the constructed vectors $q_i$ are restricted in $\calN(C)$, we name this process the \textit{Null Space Restricted GKB} (\textsf{NSR-GKB}). The pseudocode of \textsf{NSR-GKB} is shown in \Cref{alg:gGKB}.

\begin{algorithm}[htb]
	\caption{Null Space Restricted GKB (\textsf{NSR-GKB})}\label{alg:gGKB}
	\begin{algorithmic}[1]
		\Require $A\in\mathbb{R}^{m\times n}$, $C\in\mathbb{R}^{p\times n}$, $b\in\mathbb{R}^{m}$
		\State Compute $\delta_1=\|b\|_2$, \ $p_1=b/\delta_1$, 
		\State Compute $s=\calP_{\calN(C)}A^{\sT}p_1$, \ $\gamma_1=\|s\|_2$, \ $q_1=s/\gamma_1$
		\For {$i=1,2,\dots,k,$}
		\State $r=Aq_i-\gamma_i p_i$, 
		\State $\delta_{i+1}=\|r\|_2$, \ $p_{i+1}=r/\delta_{i+1}$
		\State $s=\calP_{\calN(C)}A^{\sT}p_{i+1}-\delta_{i+1}q_i$ 
		\State $\gamma_{i+1}=\|s\|_2$, \ $q_{i+1} = s/\gamma_{i+1}$
		\EndFor
		\Ensure $\{\gamma_i, \delta_i\}_{i=1}^{k+1}$, \ $\{p_i, q_i\}_{i=1}^{k+1}$  
	\end{algorithmic}	
\end{algorithm}

Note that $\calP_{\calN(C)}=\rI_{n}-C^{\dag}C$. In the computation of \textsf{NSR-GKB}, the orthonormal basis of $\calN(C)$ is not required, where instead at each step we need to compute $C^{\dag}CA^{\sT}p_{i}$, which is the most costly part. Write $\tilde{v}_i=CA^{\sT}p_{i}$, which is easy to compute. To get a good approximate to $C^{\dag}\tilde{v}_i$, we can iteratively compute the minimum 2-norm solution of the LS problem
\begin{equation}\label{ls2_gkb}
	\min_{x\in\mathbb{R}^{n}}\|Cx-\tilde{v}_i\|_2 ,
\end{equation}
which can be done efficiently by using the LSQR algorithm. In this case, \textsf{NSR-GKB} has the nested inner-outer iteration structure. If $C$ has a special structure such that its rank-revealing QR factorization is relatively easy to compute, we can first get the QR factorization of $C$ and then compute $C^{\dag}\tilde{v}_i$ directly.

The following result characterizes the structures of the two Krylov subspaces generated by \textsf{NSR-GKB}.
\medskip
\begin{proposition}\label{prop:gGKB1}
	For the \textsf{NSR-GKB} process, the generated vectors $\{q_i\}_{i=1}^{k}\subset\calN(C)$ constitute a $2$-orthonormal basis of the Krylov subspace 
	\begin{equation}\label{krylov1}
		\mathcal{K}_k(\calP_{\calN(C)}A^{\sT}A, \calP_{\calN(C)}A^{\sT}b) = \mathrm{span}\{(\calP_{\calN(C)}A^{\sT}A)^{i}\calP_{\calN(C)}A^{\sT}b\}_{i=0}^{k-1},
	\end{equation}
	and $\{p_i\}_{i=1}^{k}\subset\mathbb{R}^{m}$ constitute a $2$-orthonormal basis of the Krylov subspace 
	\begin{equation}\label{krylov2}
		\mathcal{K}_k(A\calP_{\calN(C)}A^{\sT}, b) = \mathrm{span}\{(A\calP_{\calN(C)}A^{\sT})^{i}b\}_{i=0}^{k-1}.
	\end{equation}
\end{proposition}
\begin{proof}
	To get more insights into the \textsf{NSR-GKB} process, here we give two proofs.

	The first proof is based on the property of GKB for linear compact operators \cite{caruso2019convergence}. By \Cref{prop:GKB-op}, the \textsf{NSR-GKB} is essentially the operator-type GKB of $\{\calA,b\}$, where the underlying Hilbert spaces are $\calX:=(\calN(C), \langle \cdot, \cdot \rangle_{2})$ and $\mathbb{R}^{m}$.
	Therefore, under the canonical bases, the generated vectors satisfy $q_i\in \calN(C)$ and $p_i\in\mathbb{R}^{m}$, and $\{q_i\}_{i=1}^{k}$ and $\{q_i\}_{i=1}^{k}$ are $2$-orthonormal bases of the Krylov subspaces $\mathcal{K}_k(\calA^{*}\calA,\calA^{*}\calP_{\calR(P)}b)$ and $\mathcal{K}_k(\calA\calA^{*},\calP_{\calR(P)}b)$, respectively. Since $\calA^{*}y=\calP_{\calN(C)}A^{\sT}y$ for any $y\in\mathbb{R}^{m}$, we have
	\begin{equation*}
		(\calA^{*}\calA)^{i}\calA^{*}b = (\calP_{\calN(C)}A^{\sT}A)^{i}\calP_{\calN(C)}A^{\sT}b,
	\end{equation*}
	and
	\begin{equation*}
		(\calA\calA^{*})^{i}\calP_{\calR(P)}b = (A\calP_{\calN(C)}A^{\sT})^{i}b .
	\end{equation*}
	The desired result immediately follows.

	The second proof uses the recursions \cref{GKB_op1}, which is based on a 2-orthonormal basis $\{w_{i}\}_{i=1}^{t}$ for $\calN(C)$. The standard GKB process of $\{(AW_{t}), b\}$ with recursions \cref{GKB_op1} generates two 2-orthonormal basis $\{\tilde{q}_{i}\}_{i=1}^{k}$ and $\{p_{i}\}_{i=1}^{k}$ for the two Krylov subspaces
	\begin{align*}
		& \mathcal{K}_k((AW_{t})^{\sT}AW_{t}, (AW_{t})^{\sT}b) = \mathrm{span}\{((AW_{t})^{\sT}AW_{t})^{i}(AW_{t})^{\sT}b\}_{i=0}^{k-1},  \\
		& \mathcal{K}_k(AW_{t}(AW_{t})^{\sT}, b) = \mathrm{span}\{(AW_{t}(AW_{t})^{\sT})^{i}b\}_{i=0}^{k-1},
	\end{align*}
	respectively. Using $q_{i}=W_{t}\tilde{q}_{i}$, $W_{t}W_{t}^{\sT}=\calP_{\calN(C)}$ and noticing that 
	\begin{align*}
		& \ \ \ \ W_{t} ((AW_{t})^{\sT}AW_{t})^{i}(AW_{t})^{\sT}b \\
		&= W_{t}(W_{t}^{\sT}A^{\sT}AW_{t})^{i}W_{t}^{\sT}A^{\sT}b 
		= (W_{t}W_{t}^{\sT}A^{\sT}A)^{i}W_{t}W_{t}^{\sT}A^{\sT}b \\
		&= (\calP_{\calN(C)}A^{\sT}A)^{i}\calP_{\calN(C)}A^{\sT}b,
	\end{align*}
	we immediately obtain \cref{krylov1}. Similarly, we can obtain \cref{krylov2}.
\end{proof}

Since the dimensions of $(\calN(C), \langle \cdot, \cdot \rangle_{2})$ and $\mathbb{R}^{m}$ are $\mathrm{dim}(\calN(C))$ and $m$, respectively, \Cref{prop:gGKB1} implies that \textsf{NSR-GKB} will eventually terminate at most $\min\{\mathrm{dim}(\calN(C)),m\}$ steps. The ``terminate step" of \textsf{NSR-GKB} is defined as $k_t=\min\{k: \alpha_{k+1}\beta_{k+1}=0\}$, which means that $\gamma_{i}$ or $\delta_i$ equals zero at the current step and thereby the Krylov subspace can not expand any longer. 
Suppose \textsf{NSR-GKB} does not terminate before the $k$-th iteration, that is, $\gamma_{i}\delta_{i}\neq 0$ for $1\leq i \leq k$. Then the $k$-step \textsf{NSR-GKB} process generates two $2$-orthonormal matrices $Q_{k}=(q_1,\dots,q_{k})\in \mathbb{R}^{n\times k}$ and $P_{k}=(p_1,\dots,p_{k})\in \mathbb{R}^{m\times k}$ that satisfy the following matrix-form relations:
\begin{equation}{\label{eq:GKB_matForm}} 
	\begin{cases}
		\beta_1Q_{k+1}e_{1} = b  \\
		AP_k = Q_{k+1}B_k  \\
		\calP_{\calN(C)}A^{\sT}Q_{k+1} = P_kB_{k}^{\sT}+\gamma_{k+1}q_{k+1}e_{k+1}^\sT ,
	\end{cases}
\end{equation}
where $e_1$ and $e_{k+1}$ are the first and $(k+1)$-th columns of the identity matrix of order $k+1$, and the bidiagonal matrix
\begin{equation}
	B_{k}
	=\begin{pmatrix}
		\gamma_{1} & & & \\
		\delta_{2} &\gamma_{2} & & \\
		&\delta_{3} &\ddots & \\
		& &\ddots &\gamma_{k} \\
		& & &\delta_{k+1}
		\end{pmatrix}\in  \mathbb{R}^{(k+1)\times k}
\end{equation}
has full column rank. We remark that it may happen that $\delta_{k+1}=0$, meaning that \textsf{NSR-GKB} terminates at the $k$-th step with $q_{k+1}=\mathbf{0}$. 

Now we seek the approximation to $A_{\calN(C)}^{\dag}b$ by computing the solution of \cref{NS_LS} in the Krylov subspace $\mathrm{span}\{Q_{k}\}\subset\calN(C)$. For any $x\in\mathrm{span}\{Q_{k}\}$, let $x=Q_{k}y$ with $y\in\mathbb{R}^{k}$. Using the relations \cref{eq:GKB_matForm}, we get 
\begin{align*}
   \min_{x=Q_{k}y}\|Ax-b\|_2
  = \min_{y\in\mathbb{R}^{n}}\|P_{k+1}(B_{k}y-\beta_1 e_{1})\|_2
  = \min_{y\in\mathbb{R}^{n}}\|B_{k}y-\beta_1 e_{1}\|_2 .
\end{align*}
Therefore, at the $k$-th iteration, we only need to solve the following $k$-dimensional subproblem to get the approximation:
\begin{equation}\label{x_k}
	x_k = Q_{k}y_k, \ \ \ 
	y_k=\argmin_{y\in\mathbb{R}^{k}}\|B_{k}y-\beta_1 e_{1}\|_2 .
\end{equation}
Note that before \textsf{NSR-GKB} terminates, $B_{k}$ has full column rank and the LS problem $\argmin_{y}\|B_{k}y-\beta_1 e_{1}\|_2$ always has the unique solution $y_k=B_k^{\dag}\beta e_{1}$. 
As the iteration proceeds, $x_k$ will gradually approximate the true solution of \cref{NS_LS}. The following result shows that at the terminate step, we will get the exact solution of \cref{NS_LS}.
\medskip
\begin{theorem}\label{thm:naive}
  Suppose \textsf{NSR-GKB} terminates at step $k_t$. Then the iterative solution $x_{k_t}=A_{\calN(C)}^{\dag}b$, which is the the exact minimum $2$-norm solution of \cref{NS_LS}.
\end{theorem}
\begin{proof}
	Since $x_{k_t}\in\mathrm{span}{Q_k}\subset\calN(C)$, we only need to verify that $x_{k_t}$ satisfies the two relations of \Cref{NSLS_criterion}.

	Write $x_{k_t}$ as $x_{k_t}=Q_{k_t}y_{k_t}$ and use the relation $Ax_{k_t}-b =Q_{k_t+1}(B_{k_t}y_t-\beta_1 e_{1})$. We have 
	\begin{align*}
		\calP_{\calN(C)}A^{\sT}(Ax_{k_t}-b) 
		&= \calP_{\calN(C)}A^{\sT}Q_{k_t+1}(B_{k_t}y_{k_t}-\delta_{1}e_1) \\
		&= (P_{k_t}B_{k_t}^{\sT}+\gamma_{k_t+1}q_{k_t+1}e_{k+1}^{\sT})(B_{k_t}y_{k_t}-\delta_{1}e_1) \\
		&= P_{k_t}(B_{k_t}^{\sT}B_{k_t}y_{k_t}- B_{k_t}^{\sT}\delta_{1}e_{1}) + \gamma_{k_t+1}\delta_{k_t+1}v_{k_t+1}e_{k_t}^{\sT}y_{k_t} \\
		&= \gamma_{k_t+1}\delta_{k_t+1}v_{k_t+1}e_{k_t}^{\sT}y_{k_t} \\
		&= \mathbf{0},
	  \end{align*}
	  since $\gamma_{k_t+1}\delta_{k_t+1}=0$ and $B_{k_t}^{\sT}B_{k_t}y_{k_t}=B_{k_t}^{\sT}\delta_{1}e_{1}$ due to $y_{k_t}=\argmin_{y}\|B_{k_t}y-\beta_{1}e_1\|_2$. This verifies the first relation of \Cref{NSLS_criterion}. By \Cref{prop:gGKB1} we have $x_{k_t}\in\calP_{\calN(C)}\calR(A^{\sT})=\calP_{\calN(C)}\calN(A)^{\perp}$. Let $x_{k_t}=\calP_{\calN(C)}w$ with $w\in\calN(A)^{\perp}$. For any $y\in\calN(A)\cap\calN(C)$, we have
	\begin{equation*}
		\langle x_{k_t}, y \rangle_{2} = \langle \calP_{\calN(C)}w, y \rangle_{2}
		= \langle w, \calP_{\calN(C)}y \rangle_{2} = \langle w, y \rangle_{2} = 0.
	\end{equation*}
	This verifies the second relation of \Cref{NSLS_criterion}.
\end{proof}

In the practical computation, we do not need to compute $B_{k}^{\dag}$ to get $x_{k}$ at each iteration. Instead, by exploiting the bidiagonal structure of $B_{k}$, we can design a recursive procedure to update $x_{k}$ based on the Givens QR factorization of $B_k$. This procedure follows a very similar approach proposed in \cite[Section 4.1]{Paige1982}, and we omit the derivation. Combining the \textsf{NSR-GKB} process and the update procedure, we get the following \Cref{alg:alg2} for approximating $A_{\calN(C)}^{\dag}b$. This algorithm is named the \textit{Null Space Restricted LSQR}(\textsf{NSR-LSQR}). We remark that for notational simplicity, some notations in \Cref{alg:alg2} are the same as those in \Cref{alg:glsqr}, but the readers can easily find the differences between them.

\begin{algorithm}[htb]
	\caption{Null Space Restricted LSQR (\textsf{NSR-LSQR}) for computing $A_{\calN(C)}^{\dag}b$}\label{alg:alg2}
	\begin{algorithmic}[1]
		\Require $A\in\mathbb{R}^{m\times n}$, $C\in\mathbb{R}^{p\times n}$, $b\in\mathbb{R}^{m}$
		\State \textbf{(Initialization)}
		\State Compute $\delta_1p_1=b$, $\gamma_1q_1=\calP_{\calN(C)}A^{\sT}p_1$
    	\State Set $x_0=\mathbf{0}$, $z_1=q_1$, $\bar{\phi}_{1}=\delta_1$, $\bar{\rho}_1=\gamma_1$
		\For {$i=1,2,\dots$ until convergence,}
    \State \textbf{(Applying the NSR-GKB process)}
		\State $\delta_{i+1}p_{i+1}=Aq_{i}-\gamma_ip_{i}$ 
		\State $\gamma_{i+1}q_{i+1}=\calP_{\calN(C)}A^{\sT}p_{i+1}-\delta_{i+1}q_{i}$
    \State \textbf{(Applying the Givens QR factorization to $B_k$)}
		\State $\rho_{i}=(\bar{\rho}_{i}^{2}+\delta_{i+1}^{2})^{1/2}$
		\State $c_{i}=\bar{\rho}_{i}/\rho_{i}$
    \State $s_{i}=\beta_{i+1}/\rho_{i}$
		\State$\theta_{i+1}=s_{i}\gamma_{i+1}$ 
    \State $\bar{\rho}_{i+1}=-c_{i}\gamma_{i+1}$
		\State $\phi_{i}=c_{i}\bar{\phi}_{i}$ 
    \State $\bar{\phi}_{i+1}=s_{i}\bar{\phi}_{i}$
    \State \textbf{(Updating the solution)}
		\State $x_{i}=x_{i-1}+(\phi_{i}/\rho_{i})z_{i} $
		\State $w_{i+1}=v_{i+1}-(\theta_{i+1}/\rho_{i})z_{i}$
		\EndFor
	\Ensure Approximation to $A_{\calN(C)}^{\dag}b$
	\end{algorithmic}
\end{algorithm}

To check the convergence condition of \textsf{NSR-LSQR}, here we give a stopping criterion. The idea is based on \Cref{thm:x2} and \Cref{prop:GKB-op}, which implies that \textsf{NSR-LSQR} is a Krylov subspace iterative method applied to the operator-type LS problem $\min_{x\in\calX}\|\calA x-b\|_2$. For the standard LS problem $\min_{x}\|Ax-b\|_2$, a commonly used stopping criterion is $\frac{\|A^{\sT}r_{k}\|_2}{\|A\|_{2}\|b_k\|_2}$, where $r_{k}=\|Ax_k-b\|_2$ (here we use $r_k$ and $x_k$ to denote the quantities computed by LSQR without ambiguity); see \cite[Section 6]{Paige1982}. Similarly, for the \textsf{NSR-LSQR} algorithm, we use the following relative residual norm for the stopping criterion:
\begin{equation}\label{stop}
	\frac{\|\calA^{*}r_{k}\|_2}{\|\calA\|\|b\|_2} = \frac{\|\calP_{\calN(C)}A^{\sT}r_{k}\|_2}{\|b\|_2} \leq \mathtt{tol},
\end{equation}
where $r_k=Ax_{k}-b$, and $\|\calA\|:=\max\limits_{v\in\calN(C) \atop v\neq\mathbf{0}}\frac{\|\calA v\|_{2}}{\|v\|_{2}}$. From the proof of \Cref{thm:naive} we know that $\calA^{*}r_k$ would be zero when the exact solution is obtained. Furthermore, at each iteration, we also have
\begin{align*}
	\|\calA^{*}r_{k}\|_2 = \|\calP_{\calN(C)}A^{\sT}(Ax_k-b)\|_2
	= \|\gamma_{k+1}\delta_{k+1}v_{k+1}e_{k}^{\sT}y_{k}\|_2
	= \gamma_{k+1}\delta_{k+1}|e_{k}^{\sT}y_{k}| .
\end{align*}
This means that $\|\calA^{*}r_{k}\|_2$ can be quickly obtained with almost no additional cost. The following result provides an approach for estimating $\calA$.
\medskip
\begin{proposition}\label{prop:Anorm}
	Suppose $\{w_{i}\}_{i=1}^{t}$ is an arbitrary 2-orthonormal basis of $\calN(C)$ and $W_{t}=(w_{1},\dots,w_{t})\in\mathbb{R}^{n\times t}$. Then it holds that
	\begin{equation}
		\|\calA\| = \sigma_{\max}(AW_{t}),
	\end{equation}
	which is the largest singular value of $AW_{t}$.
\end{proposition}
\begin{proof}
	For any $v\in\calN(C)$, there exist a unique $y\in\mathbb{R}^{t}$ such that $v=W_{t}y$, and $\|v\|_{2}=\|y\|_{2}$. Therefore, we have
	\begin{equation*}
		\|\calA\|=\max_{v\in\calN(C) \atop v\neq\mathbf{0}}\frac{\|\calA v\|_{2}}{\|v\|_{2}}
		= \max_{y\in\mathbb{R}^{t} \atop y\neq\mathbf{0}}\frac{\|AW_ty\|_2}{\|W_ty\|_2} 
		= \max_{y\in\mathbb{R}^{t} \atop y\neq\mathbf{0}}\frac{\|AW_ty\|_2}{\|y\|_2} 
		= \|AW_{t}\|_2 = \sigma_{\max}(AW_{t}).
	\end{equation*} 
	The proof is completed.
\end{proof}

Note from \Cref{prop:Anorm} that $\|\calA\|=\sigma_{\max}(AW_{t})$ does not depend on the choice of the 2-orthonormal basis of $\calN(C)$. To estimate $\sigma_{\max}(AW_{t})$, a very practical approach is to apply the GKB based SVD algorithm \cite{Golub1965}. Combining \cref{GKB_op1} and \cref{GKB_op2}, we can find that \textsf{NSR-GKB} generates the same $\{\gamma_i,\delta_i\}$ as that generated by the GKB of $AW_{t}$, whenever which 2-orthonormal basis $\{w_{i}\}_{i=1}^{t}$ is used. Therefore, we can use the largest singular value of $B_k$ generated by \textsf{NSR-GKB} to approximate $\sigma_{\max}(AW_{t})$, and it will not take too many iterations to get an accurate estimate.

\subsection{Two Krylov iterative methods for the LSE problem}\label{sec4.2}
Based on \Cref{alg:glsqr} and \Cref{alg:alg2}, we give two Krylov subspace based iterative algorithms for the LSE problem, which correspond to the two approaches at the end of \Cref{sec3}, respectively.
The first algorithm named \textit{Krylov Iterative Decomposed Solver-I} (\textsf{KIDS-I}) is shown in \Cref{alg:LSE1}, and the second algorithm named \textit{Krylov Iterative Decomposed Solver-II} (\textsf{KIDS-II}) is shown in \Cref{alg:LSE2}.

\begin{algorithm}[htb]
	\caption{Krylov Iterative Decomposed Solver-I (\textsf{KIDS-I}) for \cref{LSE2}}\label{alg:LSE1}
	\begin{algorithmic}[1]
		\Require $A\in\mathbb{R}^{m\times n}$, $C\in\mathbb{R}^{p\times n}$, $b\in\mathbb{R}^{m}$, $d\in\mathbb{R}^{p}$
		\State Compute $x_{1}^{\dag}=C_{A}^{\dag}d$ by \Cref{alg:glsqr}
    	\State Compute $x_{2}^{\dag}=A_{\calN(C)}^{\dag}b$ by \Cref{alg:alg2}
		\State Compute $x^{\dag}=x_{1}^{\dag}+x_{2}^{\dag}$
	\Ensure Approximate solution of \cref{LSE2}
	\end{algorithmic}
\end{algorithm}

\begin{algorithm}[htb]
	\caption{Krylov Iterative Decomposition Solver-II (\textsf{KIDS-II}) for \cref{LSE2}}\label{alg:LSE2}
	\begin{algorithmic}[1]
		\Require $A\in\mathbb{R}^{m\times n}$, $C\in\mathbb{R}^{p\times n}$, $b\in\mathbb{R}^{m}$, $d\in\mathbb{R}^{p}$
		\State Compute $\tilde{x}_{1}^{\dag}=C^{\dag}d$ by solving $\min_{x}\|Cx-d\|_{2}$
		\State Compute $\tilde{b}=b-A\tilde{x}_{1}^{\dag}$
    	\State Compute $\tilde{x}_{2}^{\dag}=A_{\calN(C)}^{\dag}\tilde{b}$ by \Cref{alg:alg2}
		\State Compute $x^{\dag}=\tilde{x}_{1}^{\dag}+\tilde{x}_{2}^{\dag}$
	\Ensure Approximate solution of \cref{LSE2}
	\end{algorithmic}
\end{algorithm}

We give a brief comparison between the above two algorithms. Generally, for large-scale problems, both the two algorithms have a nested inner-outer iteration structure: for \textsf{KIDS-II}, we need to solve \cref{ls2_gkb} at each iteration of \Cref{alg:alg2}, while for \textsf{KIDS-II}, we need to solve \cref{ls1_gkb} and \cref{ls2_gkb} at each iteration of \Cref{alg:glsqr} and \Cref{alg:alg2}, respectively. In \textsf{KIDS-I}, the computation of $x_{1}^{\dag}$ and $x_{2}^{\dag}$ can be performed simultaneously. However, in \textsf{KIDS-II}, the three steps (corresponding to lines 1--3 in \Cref{alg:LSE2}) must be performed sequentially.


\section{Numerical experiments}\label{sec5}
We present several numerical examples to illustrate the performance of the two proposed algorithms for the LSE problems. All experiments are conducted in MATLAB R2023b with double precision. It is worth noting that much of the existing literature on LSE problems lacks numerical results, primarily due to the difficulty of constructing nontrivial test problems, particularly for large-scale matrices. Based on the analysis of the LSE problems in \Cref{sec3}, we propose a procedure to construct LSE problems for testing purposes.

\paragraph*{Construct test problems.}
From the proof of \Cref{LSE_criterion}, the minimum 2-norm solution of \cref{LSE2} is $x^{\dag}=x_{1}^{\dag}+x_{2}^{\dag}$, where $x_{1}^{\dag}$ is the minimum 2-norm solutions of \cref{GLS} with $K=C$ and $L=A$, and $x_{2}^{\dag}$ is the minimum 2-norm solutions of \cref{GLS2}. With the help of the approach for constructing test problems for the GLS problems (see \cite[Section 5]{li2024new}), we can construct a test LSE problem using the following steps:
\begin{enumerate}
	\item[(1)] Choose two matrices $A\in\mathbb{R}^{m\times n}$ and $C\in\mathbb{R}^{p\times n}$. Compute $G=A^{\sT}A+C^{\sT}C$.
	\item[(2)] Compute a matrix $B\in\mathbb{R}^{n\times t}$ whose columns form a basis for $\calN(C)$.
	\item[(3)] Construct a vector $w_{1}\in\calR(G)$. Compute 
		\begin{equation}\label{x_true}
			x_{1}^{\dag} = w_1-B(B^{\sT}GB)^{-1}B^{\sT}Gw_1 .
		\end{equation}
	\item[(4)] Choose a vector $z_{1}\in\calR(C)^{\perp}$ and let $d=Cx_{1}^{\dag}+z_1$.
	\item[(5)] Construct a vector $w_{2}\in\mathbb{R}^{t}$ such that $w_{2}\perp\calN(AB)$, and choose a vector $z_2\in\calR(AB)^{\perp}$. 
	\item[(6)] Let $x_{2}^{\dag}=Bw_2$ and $b=Ax_{2}^{\dag}+z_2$.
	\item[(7)] Compute $x^{\dag}=x_{1}^{\dag}+x_{2}^{\dag}$.
\end{enumerate}
Note that the fourth step ensures that $x_{1}^{\dag}=C_{A}^{\dag}d$, while the sixth step ensures that $x_{2}^{\dag}=A_{\calN(C)}^{\dag}b$. By this construction, the minimum 2-norm solution of \cref{LSE2} is $x^{\dag}$. We remark that for large-scale matrices, computing \cref{x_true} can be extremely challenging. As a result, our experiments focus exclusively on small and medium-sized problems.

In the numerical experiments, we construct four test examples. For the first example, we set $A=D_1$, which is the scaled discretization of the first-order differential operator:
\begin{equation*}
	D_1=\begin{pmatrix}
		1 & -1 &  &  \\
		&  \ddots & \ddots &  \\
		&  &  1  & -1 \\
	\end{pmatrix} \in \mathbb{R}^{(n-1)\times n}.
\end{equation*}
The matrix $C\in\mathbb{R}^{2324\times 4486}$ named {\sf lp\_bnl2} comes from linear programming problems and is sourced from the SuiteSparse Matrix Collection \cite{Davis2011}. 
Let $w_1=(1,\cdots,1)^{\sT}\in\mathbb{R}^{n}$. We use the MATLAB built-in function \texttt{null.m} to compute a basis matrix $B$ for $\calN(C)$, and we let $w_2$ be the first column of $(AB)^{\sT}$. To obtain vectors $z_{1}$ and $z_{2}$, we we compute the projections of the random vectors \texttt{randn(p,1)} and \texttt{randn(m,1)} onto $\calR(C)^{\perp}$ and $\calR(AB)^{\perp}$, respectively.

For the second example, we set $A=D_2$, which is the scaled discretization of the second-order differential operator:
\begin{equation*}
	D_2=\begin{pmatrix}
		-1 & 2 & -1 &  &  \\
		&  \ddots & \ddots & \ddots &  \\
		&  & -1 & 2 &  & -1 \\
	\end{pmatrix} \in \mathbb{R}^{(n-2)\times n},
\end{equation*}
and the matrix $C\in\mathbb{R}^{5190\times 9690}$ named {\sf r05} comes from linear programming problems, taken from \cite{Davis2011}. We use almost the similar setting as the above for constructing the LSE problem, where the only difference is that we construct $w_1\in\mathbb{R}^{n}$ by evaluating the function $f(t)=t$ on a uniform grid over the interval $[0,1]$, that is, $w_{1}(k)=\frac{k-1}{n-1}$ for $k=1,\dots,n$.

For the third example, we choose the matrix $M\in\mathbb{R}^{3534\times 3534}$ named {\sf cage9} from \cite{Davis2011}, which arises from the directed weighted graph problem. Then we set $A=M(:, 1:2500)$ and $C=M(:, 2501:3534)$. Then we construct the LSE problem using almost the same setting as the above, where the only difference is that we construct $w_1\in\mathbb{R}^{n}$ by evaluating the function $f(t)=t^{2}$ on a uniform grid over the interval $[-1,1]$, that is, $w_{k}(k)=\left(\frac{2(k-1)}{n-1}-1\right)^2$ for $k=1,\dots,n$. We use {\sf cage9-I} and {\sf cage9-II} to denote $A$ and $C$, respectively.

For the fourth example, we choose the matrix $M\in\mathbb{R}^{9728\times 9728}$ named {\sf pf2177} from \cite{Davis2011}, which arises from the optimization problem. Then we set $A=M(:, 1:6500)$ and $C=M(:, 6501:9728)$. Then we construct the LSE problem using almost the same setting as the above, where the only difference is that we construct $w_1\in\mathbb{R}^{n}$ by evaluating the function $f(t)=\sin(2t)+3\cos(t)$ on a uniform grid over the interval $[-\pi,\pi]$, that is, $w(k)=\sin\left(\frac{4\pi(k-1)}{n-1}-2\pi\right)-3\cos\left(\frac{2\pi(k-1)}{n-1}-\pi\right)$ for $k=1,\dots,n$. We use {\sf pf2177-I} and {\sf pf2177-II} to denote $A$ and $C$, respectively.  

Several properties of the matrices in the four test examples are listed in \Cref{tab1}.

\begin{table}[h]
\caption{Properties of the test examples.}
\begin{tabular*}{\textwidth}{@{\extracolsep\fill}lcccccc}
\toprule%
& \multicolumn{3}{@{}c@{}}{$A$} & \multicolumn{3}{@{}c@{}}{$C$} \\\cmidrule{2-4}\cmidrule{5-7}%
Example & name & $m\times n$ & $\kappa(A)$ & name & $p\times n$ & $\kappa(C)$ \\
\midrule
1  & $D_1$ & $4485\times 4486$ & 2855.90 & {\sf lp\_bnl2} & $2324\times 4486$ & 7765.31 \\
2  & $D_2$ & $9688\times 9690$  & $1.68\times 10^{7}$  & {\sf r05} & $5190\times 9690$ & 121.82 \\
3  & {\sf cage9-I} & $2500\times 3534$  & 3.93 & {\sf cage9-II} & $1034\times 3534$ & 12.48 \\
4  & {\sf pf2177-I} & $6500\times 9728$  & 134.84  & {\sf pf2177-II} & $3228\times 9728$  & 44.00 \\
\botrule
\end{tabular*}\label{tab1}
\end{table}

\begin{figure}[!htbp]
	\centering
	\subfloat[]
	{\label{fig:5a}\includegraphics[width=0.49\textwidth]{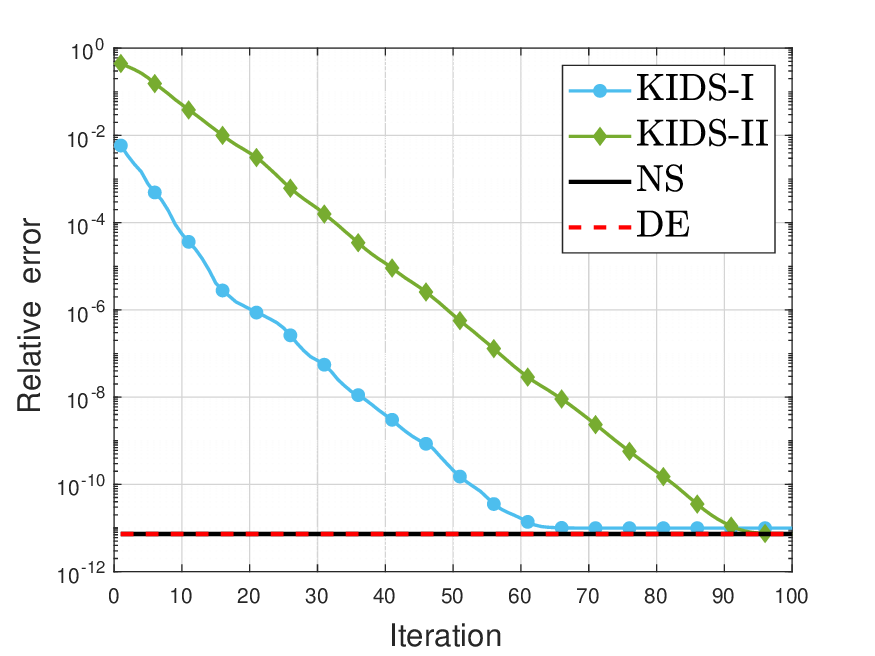}}\hspace{-0.0mm}
	\subfloat[]
	{\label{fig:5b}\includegraphics[width=0.48\textwidth]{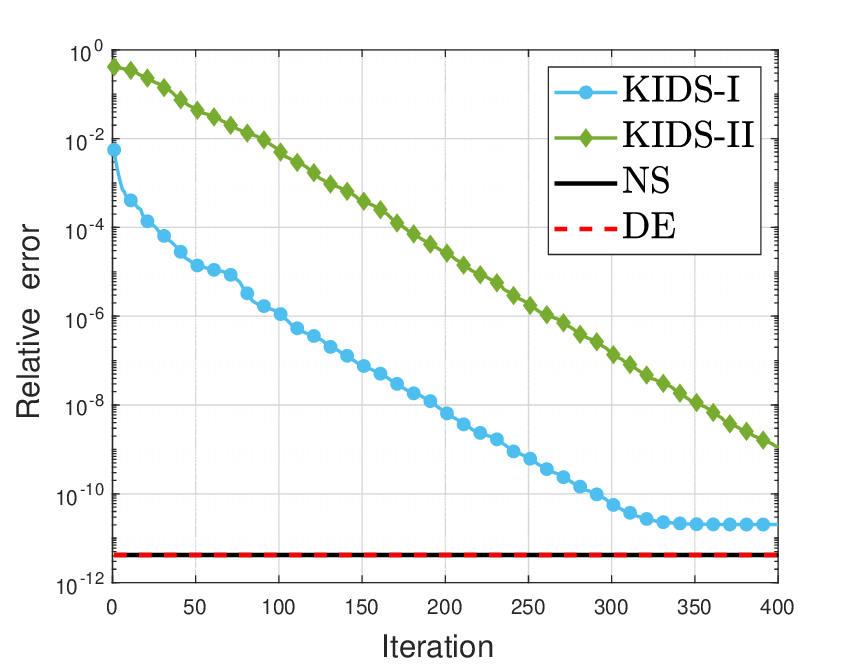}}
	\vspace{-3mm}
	\subfloat[]
	{\label{fig:5c}\includegraphics[width=0.49\textwidth]{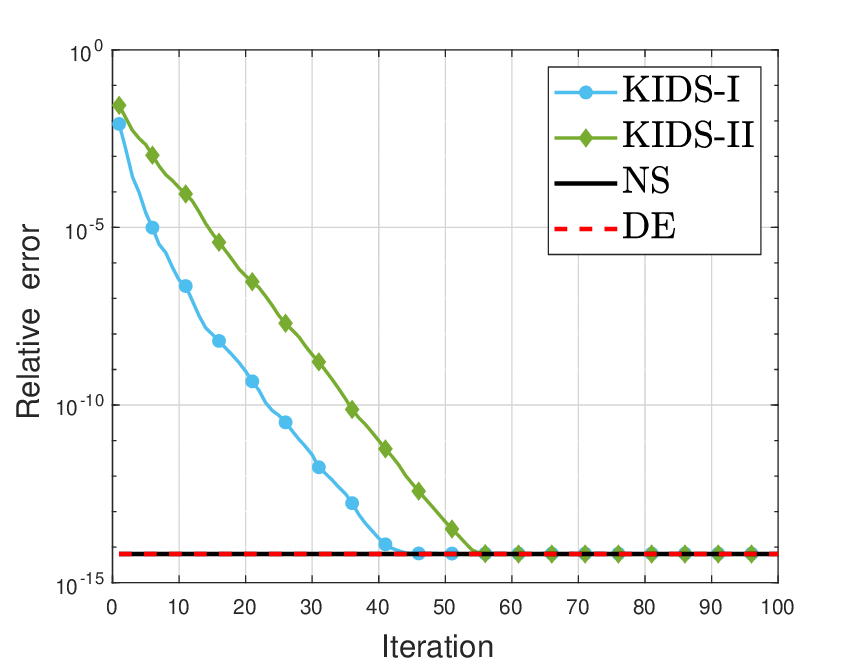}}\hspace{-0.0mm}
	\subfloat[]
	{\label{fig:5d}\includegraphics[width=0.50\textwidth]{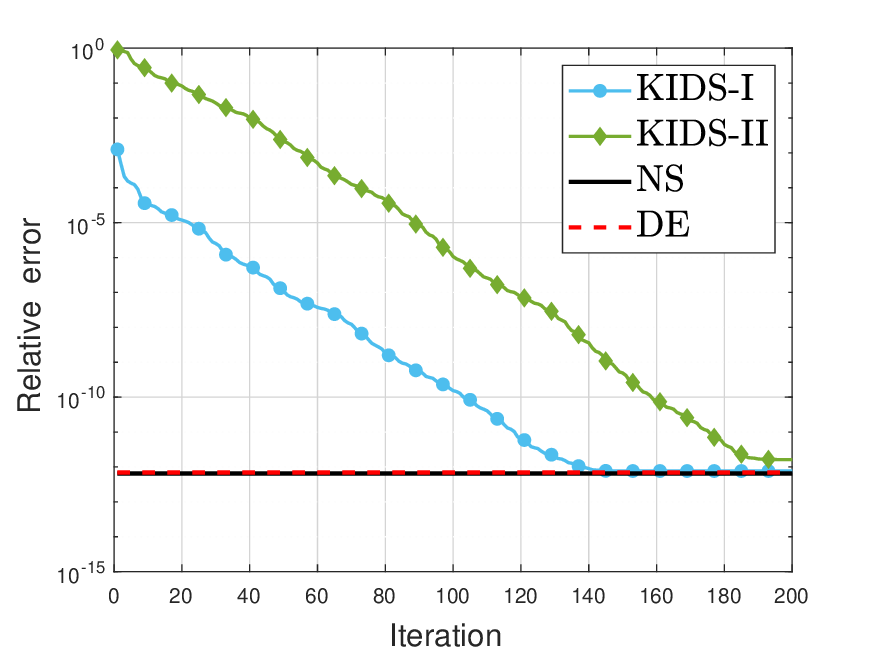}}
	\vspace{-0mm}
	\caption{The convergence history of \textsf{KIDS-I} and \textsf{KIDS-II} with respect to the true solution, where all the inner iterations are computed accurately. (a) \{$D_1$, {\sf lp\_bnl2}\}; (b)\{$D_2$, {\sf r05}\}; (c) \{{\sf cage9-I}, {\sf cage9-II}\}; (d) \{{\sf pf2177-I}, {\sf pf2177-II}\}.}
	\label{fig1}
\end{figure}

\paragraph*{Experimental results.}
In this experiment, we demonstrate the convergence behavior and the final accuracy of the approximate solutions computed by \textsf{KIDS-I} and \textsf{KIDS-II}. For comparison, we also compute two solutions using the null space method and the direct elimination method, denoted as ``NS'' and ``DE'', respectively. For the \textsf{KIDS-I} algorithm, at the $k$-th step, we compute the approximations to $x_{1}^{\dag}$ and $x_{2}^{\dag}$, respectively, which are denoted by $x_{1k}$ and $x_{2k}$. We then compute the $k$-th approximate solution of \cref{LSE0} as $x_{k}=x_{1k}+x_{2k}$. For the \textsf{KIDS-II} algorithm, we first compute $\tilde{x}_{1}^{\dag}$, which is the solution to the LS problem $\min_{x}\|Cx-d\|_2$. Then we apply \Cref{alg:alg2} to compute the approximations to $\tilde{x}_{2}^{\dag}$, where we denote the $k$-th approximation by $\tilde{x}_{2k}$. The $k$-th approximate solution of \cref{LSE0} by \textsf{KIDS-II} is $x_k=\tilde{x}_{1}^{\dag}+\tilde{x}_{2k}$. In this experiment, all the inner iterations are computed accurately.

\Cref{fig1} shows the convergence history of the two algorithms with respect to the true solution. We have three key findings. First, for both algorithms, all the approximate solutions eventually converge to the exact solution of the LSE problem, with accuracy that is almost the same as, or slightly lower than, the solutions obtained by the NS or DE methods. Second, both \textsf{KIDS-I} and \textsf{KIDS-II} exhibit a linear convergence rate for the four test problems. Since the two algorithms are based on the operator-form GKB process, we hypothesize that the convergence rate may share similarities with the LSQR algorithm. However, a more detailed investigation is needed in the future to confirm this. Third, compared with \textsf{KIDS-II}, \textsf{KIDS-I} requires fewer iterations to achieve a solution with a given accuracy. However, it is not yet clear whether this is a general property of the algorithms.

\begin{figure}[!htbp]
	\centering
	\subfloat[]
	{\label{fig:2a}\includegraphics[width=0.48\textwidth]{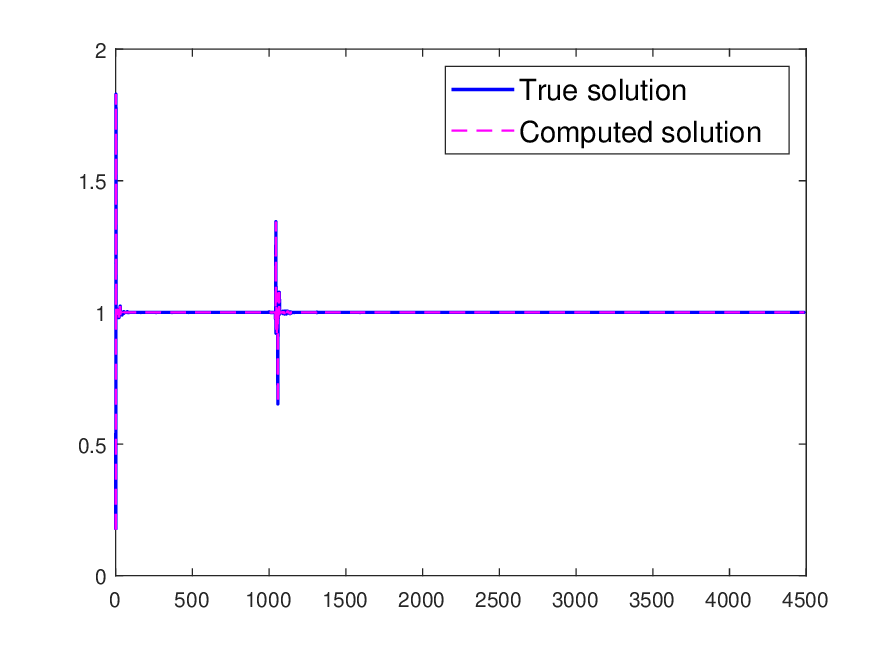}}\hspace{-0.0mm}
	\subfloat[]
	{\label{fig:2b}\includegraphics[width=0.48\textwidth]{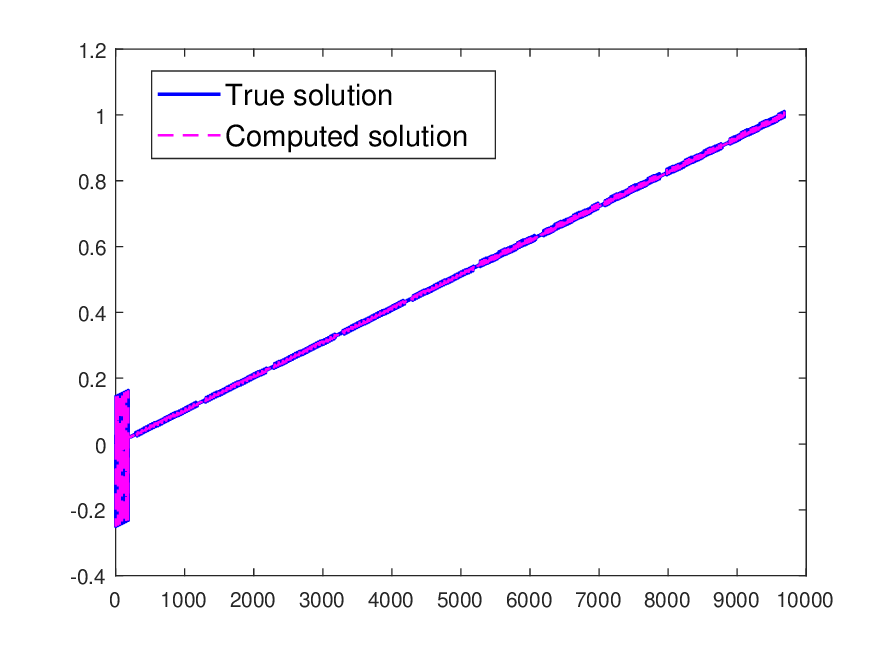}}
	\vspace{-3mm}
	\subfloat[]
	{\label{fig:2c}\includegraphics[width=0.49\textwidth]{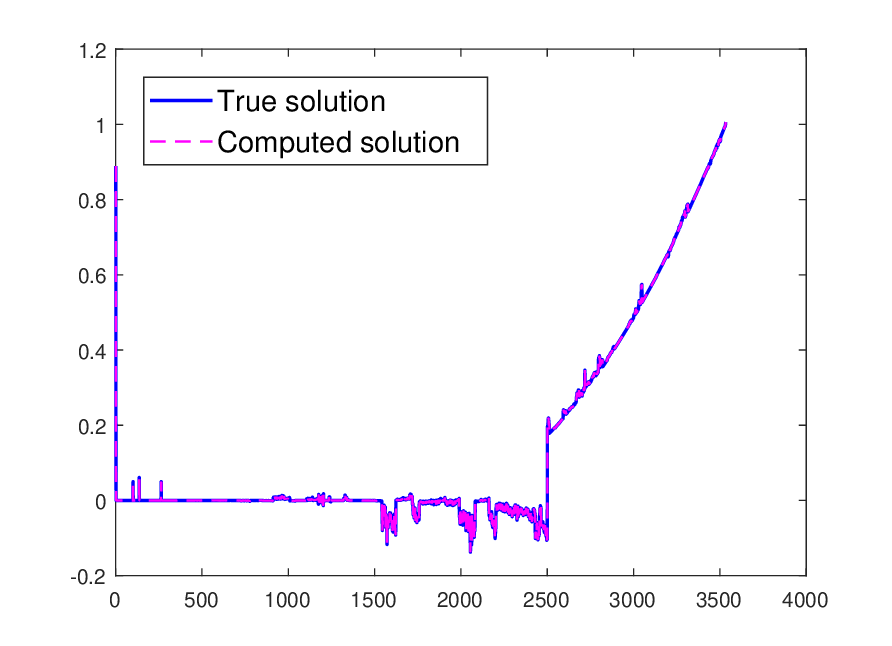}}\hspace{-0.0mm}
	\subfloat[]
	{\label{fig:2d}\includegraphics[width=0.48\textwidth]{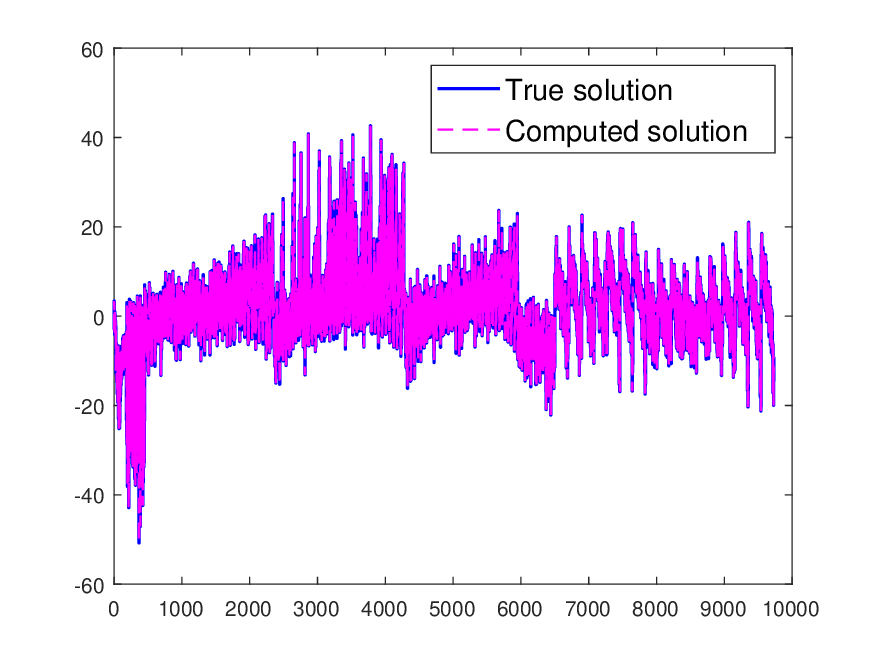}}
	\vspace{-0mm}
	\caption{Curves for the true and computed solutions obtained by \textsf{KIDS-I} at the final iteration. (a) \{$D_1$, {\sf lp\_bnl2}\}; (b)\{$D_2$, {\sf r05}\}; (c) \{{\sf cage9-I}, {\sf cage9-II}\}; (d) \{{\sf pf2177-I}, {\sf pf2177-II}\}.}
	\label{fig2}
\end{figure}

In \Cref{fig2} we plot the curve corresponding to $x_k$ computed by \textsf{KIDS-I} at the final iteration, alongside the true solution $x^{\dag}$. It is important to note that the curves corresponding to the true solutions are not smooth, as the operations used in constructing the test problems can lead to oscillating vectors. We remark that constructing a smooth true solution based on the proposed procedure for generating a test LSE problem is quite challenging. From the figure, we observe that the computed solutions closely match the true solution. These results demonstrate the effectiveness of the proposed algorithms in iteratively solving LSE problems.

\begin{figure}[!htbp]
	\centering
	\subfloat[]
	{\label{fig:3a}\includegraphics[width=0.49\textwidth]{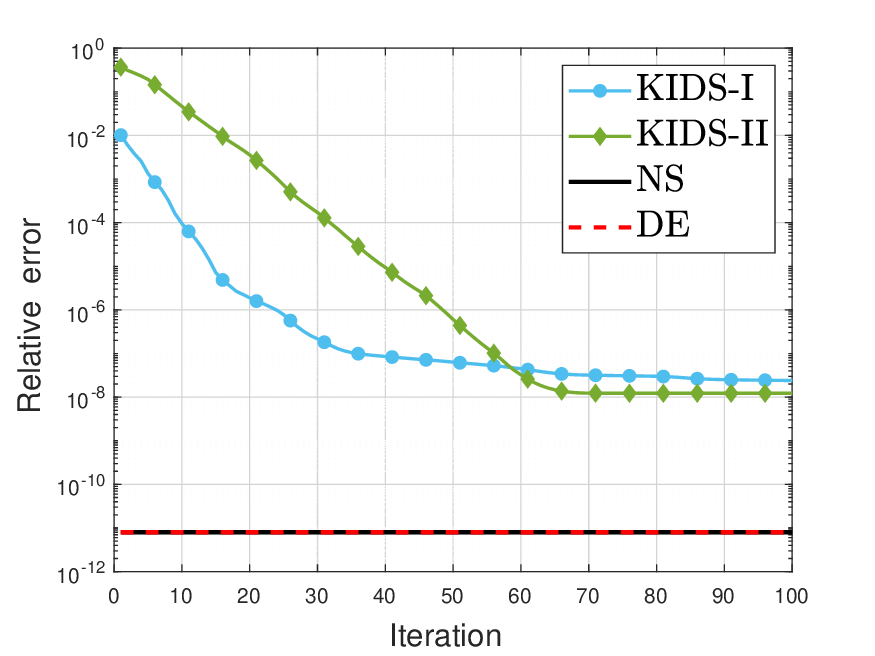}}\hspace{-0.0mm}
	\subfloat[]
	{\label{fig:3b}\includegraphics[width=0.48\textwidth]{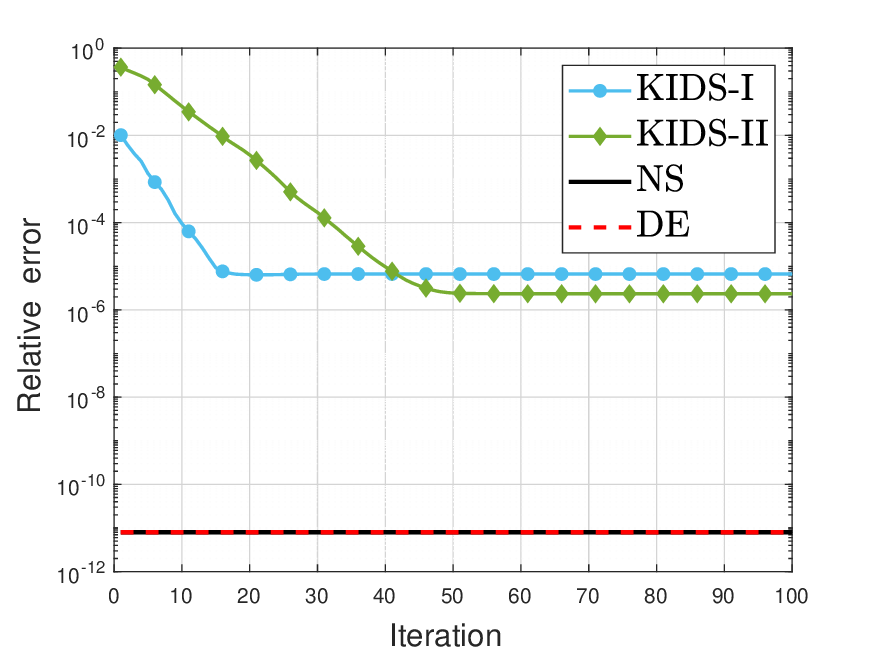}}
	\vspace{-3mm}
	\subfloat[]
	{\label{fig:3c}\includegraphics[width=0.49\textwidth]{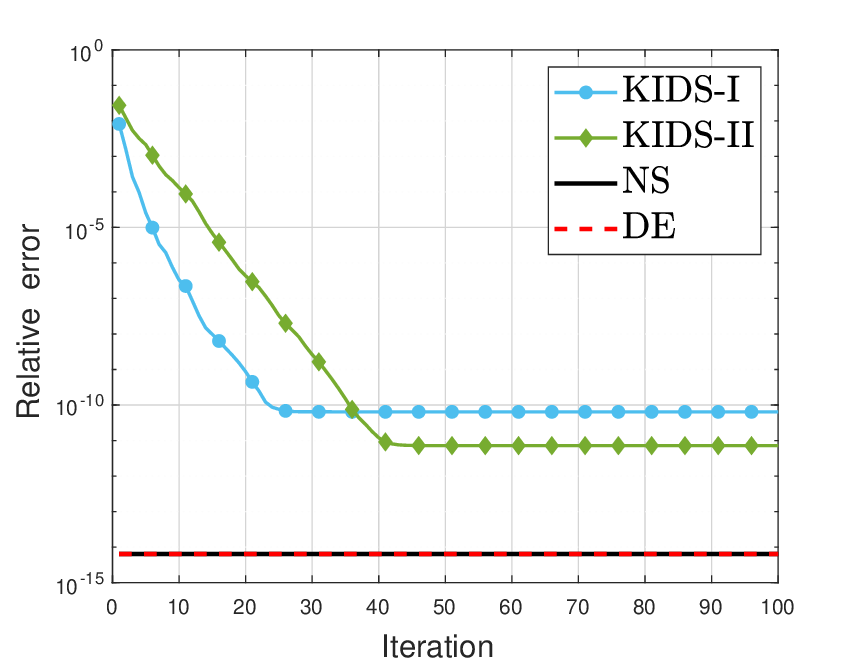}}\hspace{-0.0mm}
	\subfloat[]
	{\label{fig:3d}\includegraphics[width=0.50\textwidth]{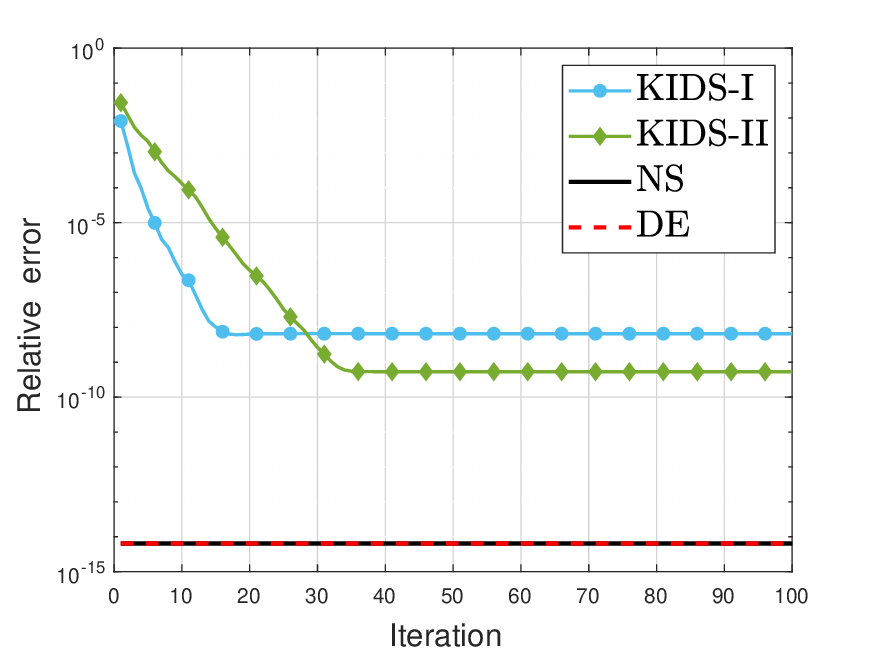}}
	\vspace{-0mm}
	\caption{The convergence history of \textsf{KIDS-I} and \textsf{KIDS-II} with respect to the true solution, where the inner iterations are approximated by solving \cref{ls1_gkb} and \cref{ls2_gkb} by LSQR with stopping tolerance $\tau$. (a) \{$D_1$, {\sf lp\_bnl2}\}, $\tau=10^{-10}$; (b)\{$D_1$, {\sf lp\_bnl2}\}, $\tau=10^{-8}$; (c) \{{\sf cage9-I}, {\sf cage9-II}\}, $\tau=10^{-10}$; (d) \{{\sf cage9-I}, {\sf cage9-II}\}, $\tau=10^{-8}$.}
	\label{fig3}
\end{figure}

In this experiment, we investigate how the inaccuracy in computing the inner iterations of both \textsf{KIDS-I} and \textsf{KIDS-II} affects the final accuracy of the approximate solutions. For \textsf{KIDS-I}, at each iteration, we use LSQR with stopping tolerance $\tau$ to iteratively solving \cref{ls1_gkb} and \cref{ls2_gkb} for approximating $x_{1}^{\dag}$ and $x_{2}^{\dag}$, respectively. For \textsf{KIDS-II}, we first compute an exact solution $\tilde{x}_{1}^{\dag}$, and then use LSQR with stopping tolerance $\tau$ to iteratively solving \cref{ls2_gkb} for approximating $\tilde{x}_{2}^{\dag}$. The stopping tolerance value $\tau$ for LSQR are set to $10^{-10}$ and $10^{-8}$. For simplicity, we only present the results for the first and third examples, as the results for the other two examples are similar.
From \Cref{fig3}, we observe that the value of $\tau$ significantly impacts the final accuracy of $x_k$, with the accuracy being approximately on the order of $\mathcal{O}(\tau)$. On the other hand, the convergence rate is not affected very much. It is important to investigate how the final accuracy of the computed solution is influenced by the value of $\tau$, especially because, for large-scale problems, it is not feasible to compute the inner iterations accurately. This aspect should be explored further in future work.

\begin{figure}[!htbp]
	\centering
	\subfloat[$\tau_2=0$]
	{\label{fig:4a}\includegraphics[width=0.48\textwidth]{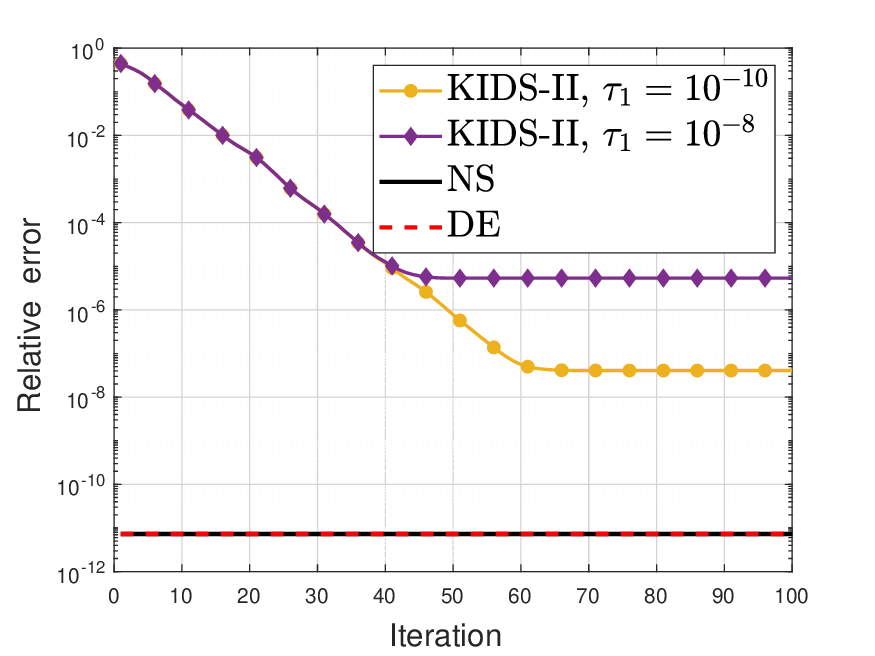}}\hspace{-0.0mm}
	\subfloat[$\tau_2=\tau_1$]
	{\label{fig:4b}\includegraphics[width=0.49\textwidth]{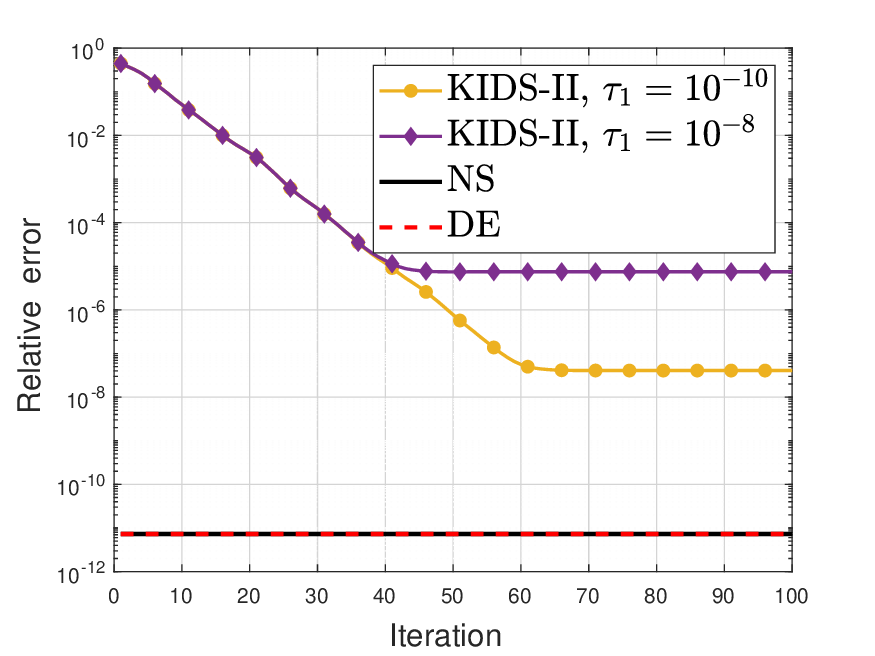}}
	\vspace{-0mm}
	\caption{The convergence history of \textsf{KIDS-II} with respect to the true solution, where $\tilde{x}_{1}^{\dag}=\argmin_{x}\|Cx-d\|_2$ is computed by LSQR with stopping tolerance $\tau_1$, and the inner iteration is approximated by solving \cref{ls2_gkb} by LSQR with stopping tolerance $\tau_2$. The test example is \{$D_1$, {\sf lp\_bnl2}\}.}
	\label{fig4}
\end{figure}

In this experiment, we explore how the solution accuracy of $\tilde{x}_{1}^{\dag}=\argmin_{x}\|Cx-d\|_2$ influences the final accuracy of $x^{\dag}$. To obtain an approximate $\tilde{x}_{1}^{\dag}$, we apply the LSQR algorithm to solve $\min_{x}\|Cx-d\|_2$ with a stopping tolerance set to $\tau_1$. The inner iteration is approximated by solving \cref{ls2_gkb} using the LSQR algorithm, with the stopping tolerance set to $\tau_2$. We only show the experimental results for the first example, as the results for the other examples are similar. First, we set $\tau_1=10^{-10}$ and $\tau_{1}=10^{-8}$, respectively, and set $\tau_{2}=0$, meaning that we compute the inner iteration accurately. The convergence history is shown in \Cref{fig:4a}. We observe that an inaccurate $\tilde{x}_{1}^{\dag}$ affects the final accuracy of $x^{\dag}$, even when the inner iterations are computed accurately. Second, we set $\tau_1=\tau_2=10^{-10}$ and $\tau_{1}=\tau_2=10^{-8}$, respectively. The convergence history is shown in \Cref{fig:4b}. We observe that when the inner iterations are performed with the same accuracy as $\tilde{x}_{1}^{\dag}$, then the final accurately of $x^{\dag}$ is comparable to the accuracy achieved when the inner iterations are computed accurately. Since the solution error of $\tilde{x}_{1}^{\dag}$ may be amplified in the subsequent computation, analyzing the impact of the inaccuracy in $\tilde{x}_{1}^{\dag}$ on the final accuracy of $x^{\dag}$ is more complex than simply analyzing the inner iterations. This also implied that \textsf{KIDS-II} can be more susceptible to computational errors than \textsf{KIDS-I}. A systematic comparison of the two algorithms and their susceptibility to computational errors will be explored in future work.

\section{Conclusion and outlook}\label{sec6}
In this paper, we have introduced a novel approach to solving the LSE problems by reformulating them as operator-type LS problems. This perspective allows us to decompose the solution of the LSE problem into two components, each corresponding to a simpler operator-based LS problem. We have derived two types of decomposed-form solutions, and building on the decompositions, we have developed two Krylov subspace based iterative methods that efficiently approximate the solution without relying on matrix factorizations. The two proposed algorithms, named \textsf{KIDS-I} and \textsf{KIDS-II}, follow a nested inner-outer structure, where the inner subproblem can be computed iteratively. We have proposed an approach to construct the LSE problems for testing purposes, and used several test examples to demonstrate the effectiveness of the algorithms.

The primary computational bottleneck of the proposed algorithms is the computation of the inner iteration. Since constructing very large-scale test examples is challenging, we have limited our numerical experiments to small and medium-sized matrices. In the future, we will explore additional theoretical and computational strategies to improve the efficiency of the inner iteration as well as construct larger-scale test problems to further assess the performance of the two algorithms.










\section*{Declarations}
\begin{itemize}
\item Competing Interests: Not applicable
\end{itemize}

\bibliography{references}

\end{document}